\documentclass[11pt]{amsart}
 \usepackage{hyperref}
  \newcommand{\LE}{\mathcal{LE}}
   \renewcommand{\H}{\mathcal{H}}
  \newcommand{\B}{\mathcal{B}}
 \newcommand{\cB}{\mathcal{B}}

\newcommand{\cD}{\mathcal{D}}

\newcommand{\G}{\mathcal{G}}
\newcommand{\cZ}{\mathcal{Z}}

\newcommand{\gG}{\Gamma}
\newcommand{\C}{\mathbb{C}}
\newcommand{\T}{\mathbb{T}}
\newcommand{\Q}{\mathbb{Q}}
\newcommand{\R}{\mathbb{R}}
\newcommand{\E}{\mathbb{E}}
\newcommand{\N}{\mathbb{N}}

\newcommand{\Z}{\mathbb{Z}}

\newcommand{\norm}[1]{\left\Vert #1\right\Vert}
\newcommand{\nnorm}[1]{\lvert\!|\!| #1|\!|\!\rvert}

\theoremstyle{plain}
\newtheorem{theorem}{Theorem}[section]
\newtheorem{lemma}[theorem]{Lemma}
\newtheorem{proposition}[theorem]{Proposition}
\newtheorem*{theoremA}{Theorem A}
\newtheorem*{theoremB}{Theorem B}
\newtheorem*{theoremC}{Theorem C}

\newtheorem*{conjectureA}{Conjecture A}
\newtheorem*{conjectureB}{Conjecture B}

\newtheorem*{theoremA'}{Theorem A'}
\newtheorem*{theoremB'}{Theorem B'}
\newtheorem*{theoremC'}{Theorem C'}

\newtheorem*{theorem*}{Theorem}
\newtheorem*{Correspondence1}{Furstenberg Correspondence Principle~(\cite{Fu1}, \cite{Be1})}

\newtheorem{corollary}[theorem]{Corollary}

\theoremstyle{definition}
\newtheorem{definition}[theorem]{Definition}
\newtheorem{example}{Example}

\theoremstyle{remark}

\newtheorem*{remark}{Remark}
\newtheorem*{remarks}{Remarks}

\begin{document}

\title{A Hardy field extension of Szemer\'edi's theorem} \author{Nikos
  Frantzikinakis}
\address[Nikos  Frantzikinakis]{Department of Mathematics\\
  University of Memphis\\
  Memphis, TN \\ 38152 \\ USA } \email{frantzikinakis@gmail.com}

\author{M\'at\'e Wierdl}
\address[M\'at\'e Wierdl]{Department of Mathematics\\
  University of Memphis\\
  Memphis, TN \\ 38152 \\ USA } \email{wierdlmate@gmail.com}

\begin{abstract}
  In 1975 Szemer\'edi proved that a set of integers of positive upper
  density contains arbitrarily long arithmetic progressions. Bergelson
  and Leibman showed in 1996 that the common difference of the
  arithmetic progression can be a square, a cube, or more generally of
  the form $p(n)$ where $p(n)$ is any integer polynomial with zero
  constant term.  We produce a variety of new results of this type
  related to sequences that are not polynomial. We show that the
  common difference of the progression in Szemer\'edi's theorem can be of the form
  $[n^\delta]$ where $\delta$ is
  any positive real number and $[x]$ denotes the integer part of $x$.
  More generally, the common difference can be of the form $[a(n)]$
  where $a(x)$ is any function that is a member of a Hardy field and satisfies
  $a(x)/x^k\to \infty$ and $a(x)/x^{k+1}\to 0$ for some non-negative
  integer $k$.
  %%This allows us for example to deal with functions that
  %%can be constructed by a finite combination of the ordinary
  %%arithmetical symbols, the real constants, the real variable $x$, and
  %%the functional symbols $\exp$ and $\log$, and satisfy the previous
  %%growth assumptions.
  The proof combines a new structural result for Hardy sequences,
  techniques from ergodic theory, and some recent equidistribution
  results of sequences on nilmanifolds.
\end{abstract}

\thanks{The first  auth\textbf{o}r was partia\textbf{l}ly supported by NSF \textbf{g}r\textbf{a}nt
  DMS-0701027 and the second by   DMS-0801316.}



\maketitle
%% \centerline{\today}

\setcounter{tocdepth}{1}%%Only sections appear in Contents
\tableofcontents
\section{Introduction and main results}
\subsection{Introduction}
In 1975 Szemer\'edi~(\cite{Sz}) answered a long standing question of
Erd\"{o}s and Tur\'an (1936, \cite{ET}), showing that a set of
integers of positive upper density\footnote{If $\Lambda\subset \Z$,
  the {\it upper density} of $\Lambda$ is the number
  $\bar{d}(\Lambda)=\limsup_{N\to\infty}|\Lambda\cap \{-N,\ldots,N\}|/(2N+1)$.
  If the previous limit exists we call it the {\it
    density} of $\Lambda$ and denote it by $d(\Lambda)$.} contains
arbitrarily long arithmetic progressions, or equivalently, for every
$\ell\in\N$, patterns of the form
\begin{equation}\label{E:d}
  \{m, m+d,m+2d,\ldots, m+\ell d\}
\end{equation}
for some $m\in \Z$ and $d\in \N$. This result has been very influential,
several different proofs and extensions have been found, and the tools
developed in the process led to applications in
several diverse  fields, that include
combinatorics, number theory, harmonic analysis, ergodic theory, and
theoretical computer science.

In this article we are interested in obtaining refinements of
Szemer\'edi's theorem by restricting the scope of the common
difference $d$.
%%We are by no means the first to deal with this
%%problem. In fact,
 During the last thirty years several related refinements
have been obtained, most notably a result of Bergelson and Leibman
(\cite{BL}), who showed that $d$ can be taken to
%% belong to any IP set meaning a set that consists of distinct sums
%% (with distinct entries of some infinite set) of some infinite set
%% (\cite{FuK2}), and $d$ can be a square or more generally
be of the form $p(n)$ where $p$ is any
non-constant integer polynomial with $p(0)=0$. This  had been
previously established for $\ell=1$ by Sark\"ozy (\cite{Sa}) and
Furstenberg (\cite{Fu2}).  More examples, related to IP sets,
generalized polynomials, polynomials with non-zero constant term, and
the set of prime numbers, can be found in \cite{FuK2}, \cite{BM},
\cite{M}, \cite{BHM}, \cite {Fr}, \cite{FHK}.  All these results were
obtained using (in addition to other tools) methods that emerged from
the pioneering paper of Furstenberg~(\cite{Fu1}), where ergodic theory
was used to give a new proof of Szemer\'edi's theorem.
%% (\cite{Sa} for $k=1$, \cite{BL} for general $k$), or, for $k=1,2$
%% of the form $p-1$ or $p+1$ where $p$ %%is prime (\cite{Sa} for $k=1$, \cite{FHK} for $k=2$). Several other examples related to generalized %%polys.

We shall produce a variety of new examples given by sequences that are
not polynomial, and range from simply defined to rather exotic looking.
For example, we shall show that the common difference $d$ in
\eqref{E:d} can be taken to be of the form
\begin{equation}\label{E:examples1}
  [n^{\sqrt{2}}], \ [n\log n], \ [\sqrt{3}\ \! n^{5/2}+n\log n],
  \ [n^2/\log\log n],  \ \Big[\sqrt{n^{2008}+(\log n)^{2/3}}+n^2e^{-\sqrt[3]{\log n}}\Big],
\end{equation}
where $[x]$ denotes the integer part of $x$, or the form
\begin{equation}\label{E:examples2}
 [\log(n!)], \quad
  [\log(\Gamma(n^{3/2}))],\quad [n^2\sin(1/\log{n})], \quad [n^{5/2}
  \zeta(n)],\quad [n^{1+1/n}\text{Li}(n)],
\end{equation}
where $\Gamma$ is the Gamma function, $\zeta$ is the Riemann zeta
function, and $\text{Li}$ is the logarithmic integral function
(defined by $\text{Li}(x)=\int_2^x 1/\log t\ dt$).

%% The previous list of sequences may look rather random, so let us
%% pinpoint some properties they share that will be very helpful for
%% our purposes.  Firstly, they all have non-polynomial growth, this
%% is helpful since there is an abundance of examples %%of polynomial sequences that we would like to avoid (see Sections~\ref{SS:ergodic} and \ref{SS:single}).
%% On the other hand, the non-polynomial structure of these sequences
%% makes them hard to work with, since %%there are no general techniques for studying non-polynomial refinements of Szemer\'edi's theorem. We'll %%see how we will overcome this problem later.  A second helpful thing is that all the previous sequences %%were constructed by evaluating smooth functions on the integers.
%% Unfortunately, these two good properties are not sufficient by
%% themselves; we can easily construct %%sequences that satisfy them but consist only of odd integers, a bad thing for our purposes.

A more illuminating (but incomplete) description of the class of
functions for which our result apply is as follows: By $\LE$ we denote
the collection of \emph{logarithmico-exponential functions} of Hardy
(\cite{Ha1}, \cite{Ha2}), consisting of all functions that can be
constructed using the real constants, the functions $e^x$ and
$\log{x}$, and the operations of addition, multiplication, division,
and composition of functions, as long as the functions constructed are
well defined for large $x$.  We shall show that if $a\in \LE$, then the
common difference $d$ in \eqref{E:d} can have the form $[a(n)]$  as long as $a$ satisfies
the growth condition $x^k
\prec a(x) \prec x^{k+1}$ for some non-negative integer $k$. The
examples in \eqref{E:examples1} are of this type.

In fact, our result applies to the much larger class of functions that
belong to some Hardy field (a notion first introduced by Bourbaki
(\cite{Bou})) and satisfy the previous growth restrictions. This will
enable us to deal with the sequences in \eqref{E:examples2} as well.
%% , p. 107):
\begin{definition}
  Let $B$ be the collection of equivalence classes of real valued
  functions $a(x)$ defined on some half line $(u,\infty)$, where we
  identify two functions if they agree for all large $x$.\footnote{The
    equivalence classes just defined are often called \emph{germs of
      functions}. We choose to use the word function when we refer to
    elements of $B$ instead, with the understanding that all the
    operations defined and statements made for elements of $B$ are
    considered only for sufficiently large values of $x\in \R$.}  A
  \emph{Hardy field} is a subfield of the ring $(B,+,\cdot)$ that is
  closed under differentiation. By $\H$ we denote the \emph{union of all
  Hardy fields}.
\end{definition}
Hardy fields have been used to study solutions of differential
equations (\cite{Bo1}, \cite{Bo2}, \cite{Bo2c}, \cite{R1}, \cite{R2}),
difference and functional equations (\cite{Bo2a}, \cite{Bo2b}),
properties of curves in $\R^2$ (\cite{Dr}), equidistribution results
of sequences on the torus (\cite{Bo3}), and convergence properties of ergodic averages
 (\cite{BKQW}).  We
collect some results that illustrate the richness of $\H$:

%%\begin{itemize}
$\bullet$ $\H$ contains $\LE$ and  anti-derivatives of elements of $\LE$.

$\bullet$ $\H$ contains several other functions not in $\LE$, like the
functions $\Gamma(x)$, $\zeta(x)$,  $\sin{(1/x)}$.

$\bullet$ If $a\in \LE$ and $b\in\H$, then  there exists a Hardy field
containing both $a$ and $b$.

$\bullet$ If $a\in\LE$, $b\in \H$, and $b(x)\to \infty$, then $a\circ
  b\in \H$.

 \ \ \  If $a\in\LE$, $b\in \H$, and $a(x)\to \infty$, then $b\circ a\in \H$.

$\bullet$ If $a$ is a continuous function that is algebraic over some
  Hardy field, then $a\in\H$.
%%\end{itemize}

%% $\bullet$ $\H$ contains solutions to several functional and
%% differential equations not in $\LE$, like solutions of
%% $f(f(x))=e^x$ and all solutions of $y''-xy=0$.

We mention some basic properties of elements of $\H$ relevant to our study. Every element of $\H$ has eventually constant sign
(since it has a multiplicative inverse).
 Therefore, if  $a\in \H$, then $a$ is eventually monotone (since $a'$ has eventually constant sign), and  the limit
$\lim_{x\to \infty} a(x)$ exists (possibly infinite).  Since  for every two
functions $a\in \H, b\in \LE$ $(b\neq 0)$, we have $a/b\in \H$, it follows that   the asymptotic growth ratio
$\lim_{x\to \infty}a(x)/b(x)$ exists (possibly infinite). This last property is key,
since it will often justify our use of  L'Hospital's rule.
\emph{We are going to freely use all these properties without any further explanation in the sequel.}

\subsection{Results in combinatorial language}
%% If it is a
%%nonzero constant we say that $a,b$ \emph{have the same growth rate}
%%and write $a(x)\asymp b(x)$.

The following is our main result:
\begin{theoremA}
  Let $a\in \H$ satisfy $x^k\prec a(x) \prec x^{k+1}$ for some non-negative
  integer $k$. Let $\ell\in \N$.

  Then  every $\Lambda\subset\mathbb{Z}$ with
  $\bar{d}(\Lambda)>0$ contains arithmetic progressions of the form
  \begin{equation}\label{E:theoremA}
    \{m, m+[a(n)],m+2[a(n)],\ldots, m+\ell[a(n)]\}
  \end{equation}
  for some $m\in \Z $ and $n\in\N$ with $[a(n)]\neq 0$.
\end{theoremA}
\begin{remarks}

$\bullet$ For $\ell=1$, Theorem~A can be easily deduced from the equidistribution
    results in \cite{Bo3} and the spectral theorem (see \cite{BKQW}
    for details).

$\bullet$ Our assumption can also be stated in the following equivalent form: $a\in \H$ has  polynomial
 growth  and is \emph{not} of the form $cx^k+b(x)$ for some non-negative integer $k$, non-zero real number $c$, and $b\in\H$ that satisfies $b(x)/x^k\to 0$. So it is functions like $x^2+\log{x}$ or  $\sqrt{2}x^3+ x\log x$ that our present methods do not allow us to handle.

$\bullet$ The assumption that $a(x)$ has  polynomial growth
 is essential if one wants to have sufficient conditions
    that depend only on the growth of the function $a(x)$.\footnote{A result mentioned in
      \cite{Bo3} suggests the possibility that for every $a\in \H$ of
      \emph{super-polynomial growth} there exists $b\in \H$ of the same
      growth, that is, the limit of $b/a$ is a non-zero real constant,
      such that  $b(n)$ is an odd integer for every $n\in\N$. If this is the
      case, then no growth assumption on elements of $\H$ with super-polynomial
       growth will be sufficient for our purposes.}  On the other hand,
    the precise assumptions on $a(x)$ in Theorem~A can probably be
    relaxed (see Conjecture~A in Section~\ref{SS:conjectures}); it certainly is possible for $\ell=1$ (see
    Theorem~C).

%%      $\bullet$ If $a(x)$ is a non-zero polynomial with real coefficients and $a(0)=0$, then  it can be
%% shown using \cite{BL} that the conclusion of  Theorem~A holds. If $a(0)\neq 0$  the situation is more delicate: If $a\in \Z[x]$,
%% then Theorem~A holds if and only if the range of $a(x)$ contains multiples of every positive integer (\cite{Fr}).
%%If $a(x)\neq cp(x)+b$ for every $b,c\in \R$ and $p\in\Z[x]$, then it can be shown that Theorem~A holds.
%%In the remaining cases an easy to state  necessary condition is harder to come by (see Section~\ref{SS:single} for more  details).
$\bullet$
  An immediate corollary is the following coloristic result,
    which we do not see how to prove without using Theorem~A: If $a\in
    \H$ satisfies the growth condition of Theorem~A, then  every
    finite coloring of the integers has a monochromatic
    arithmetic progression of the form \eqref{E:theoremA}.

$\bullet$ Although our result applies to rather exotic sequences, like
    the sequences mentioned in the examples \eqref{E:examples1} and
    \eqref{E:examples2}, simply defined sequences like
    $[n^{\sqrt{5}}]$ seem to be almost as hard to deal with as the
    general case.

$\bullet$ Unlike the case where $a(n)$ is a polynomial with zero constant term, it is not true that
    $\Lambda\cap (\Lambda-[a(n)])\neq \emptyset$ for a set of $n\in\N$
    with bounded gaps. To see this, take $\Lambda=2\Z$, $a(n)=\log n$,
    and notice that $[a(n)]$ takes odd values for every $n\in
    [2^{2l+1},2^{2l+2})$ for every $l\in\N$. With a bit more effort one
    can show that we have the same problem for every $a\in\H$ that  satisfies
    the growth assumptions of  Theorem~A.
\end{remarks}
To prove Theorem~A we first  use the correspondence
principle of Furstenberg  (see Section~\ref{SS:ergodic}) to
translate it into a statement about
multiple recurrence in ergodic theory. The ergodic
method used to prove Szemer\'edi's theorem (\cite{Fu1}) and its
polynomial extension (\cite{BL}) does not seem to apply\footnote{The
  main problem appears when one deals with distal systems. Unlike the
  case of a polynomial with zero constant term, for $a\in \H$
  satisfying $x^k\prec a(x) \prec x^{k+1}$ for some non-negative
  integer $k$, successive applications of the operation
  $[a(n+m)]-[a(n)]-[a(m)]$, $m\in\N$, lead eventually to non-zero
  constant sequences (in $n$) which is a problem when one tries to
  prove the corresponding coloristic (van der Waerden type) result.},
so we  use a different method instead.  Our argument splits into
three parts:

$(i)$ As it turns out, dealing with the full sequence $[a(n)]$ greatly
complicates our study, in particular step $(iii)$ below. Instead,
we show that the range of $[a(n)]$ contains some suitably chosen
polynomial patterns of fixed degree
%%\footnote{The choice of the patterns will also
 %% depend on the set $\Lambda\subset \Z$ that we are given.}
(Proposition~\ref{P:polypattern}), and we work with this collection of
patterns henceforth.
%%\footnote{The choice of the patterns will also
  %%depend on the set $\Lambda\subset \Z$ that we are given.}.
  To obtain these patterns we use the Taylor expansion of the function
  $a(x)$. Since some derivative of $a(x)$ vanishes at infinity,  it makes sense
  to expect
  (but is non-trivial to verify) that the range of $[a(n)]$ has a rich supply of polynomial progressions
  of fixed degree.

$(ii)$ For the polynomial patterns found in $(i)$, we study the
naturally associated multiple ergodic averages, and show that the
nilfactor of the system controls their limiting behavior
(Proposition~\ref{P:seminorm2}). As a consequence, we reduce our problem to
 establishing a certain multiple recurrence property for
nilsystems. This reduction to nilsystems step is carried out using a rather
cumbersome application of the by now standard polynomial exhaustion technique (PET induction);
we use it  to eliminate some undesirable constants and majorize our multiple ergodic averages
 by some polynomial ones that we know how to control.

$(iii)$ We verify the multiple recurrence property for nilsystems
by
comparing the multiple ergodic averages along the polynomial patterns of
part $(i)$ with some easier to handle averages that can be estimated
using Furstenberg's classical multiple recurrence result.  To carry out
the comparison step we need
 an equidistribution result on nilmanifolds
(Proposition~\ref{P:equidistributiongeneral}).  Because our
polynomial patterns consist of finite polynomial blocks rather than a
single infinite polynomial sequence, the result we
need does not seem to follow from the available qualitative
equidistribution results of polynomial sequences on nilmanifolds.
Instead, we adapt a quantitative equidistribution result that was
recently obtained by Green and Tao (\cite{GT}).

To give an example of the polynomial patterns we are led to
consider, let us look at the case of the sequence $[a(n)]$ where
$a\in \H$ satisfies $x\prec a(x) \prec x^{2}$. In this case, we can
show that the range of the sequence $[a(n)]$ contains certain
arithmetic progressions and derive Theorem~A from the following
result (that we find of interest on its own):
\begin{theoremB}
Suppose that for every $r,m\in \N$ the set $S\subset \Z$
  contains patterns of the form $$\{r(c_{r,m}+mn)\colon 1\leq
  n\leq N_{r,m}\}$$ where $c_{r,m}, N_{r,m}$ are integers with $\lim_{m\to\infty}N_{r,m}=\infty$
  for every fixed $r\in \N$.

  Then for every $\ell \in \N$ every
  $\Lambda\subset\mathbb{Z}$ with $\bar{d}(\Lambda)>0$ contains
  patterns of the form
$$
\{m, m+s,r+2s,\ldots, m+\ell s\}
$$
for some $m\in \Z$ and non-zero $s\in S$.
%% If $c_m,N_m\in\N$ with $N_m\to +\infty$, then
%%$$
%%S=\{c_m+mn\colon 1\leq n\leq N_m, m\in \N\}
%%$$
%%is a set of multiple recurrence.
\end{theoremB}
\begin{remarks}
$\bullet$ See Theorem~\ref{T:recurrencepoly} for a result that deals with more general \emph{polynomial} progressions.

  $\bullet$
   If $c_m=0$ for infinitely many $m\in\N$, the result
  follows easily from a finitistic version of Szemer\'edi's theorem. Such
  an easy derivation doesn't seem to be possible when we have no
  (usable) control over the constants $c_m$.
%%  $\bullet$ As it will become clear from the proof, the assumption
%%  ``for every $m\in \N$ ...'' can be relaxed to ``for a set of $m\in
%%  \N$ with positive upper density ...''.
\end{remarks}
To prove Theorem~A we shall need a generalization of Theorem~B
 that deals with more complicated
polynomial patterns (Theorem~\ref{T:recurrencepoly}).  In order to illustrate some of the ideas needed
to prove Theorem~A in their simplest form, we choose to present the proof of
Theorem~B separately.

Next we mention an improvement of our main result for $\ell=1$. This result was first obtained (but never published) several
    years ago by Boshernitzan using a method different than ours.

\begin{theoremC}\label{T:singlehardy}
  Let $a\in\H$ have polynomial growth and suppose that  $|a(x)-cp(x)|\to \infty $ for every $p\in
  \Z[x]$ and
  $c\in \R$.

  Then  every $\Lambda\subset\mathbb{Z}$ with $\bar{d}(\Lambda)>0$,
  contains  $x,y\in \Lambda$ that satisfy  $y-x=[a(n)]$ for some
  $n\in\N$ with $[a(n)]\neq 0$.
\end{theoremC}
%%\begin{remark}
%% The conclusion of Theorem~C also holds for functions $a\in \H$ of
%%  the form   $a=p+b$ where $p\in\R[x]$ with $p(0)=0$ and $b$ satisfies
%%    $x^k\prec b(x) \prec x^{k+1}$ for some non-negative integer $k$.
%%    This is not hard to check using  Theorem~C.
%%\end{remark}
The proof of Theorem~C is rather different (and much easier) than the
proof of our main result (Theorem~A). In order not to digress from our
main objective we give it in the Appendix.

Although we were not able to prove Theorem~A under the more relaxed
assumptions of Theorem~C, we believe that the corresponding stronger
statement should be true (see Conjecture~A below).
%% \begin{conjecture}
%%   The conclusion of Theorem~A remains true for every $\ell\in\N$
%%   under the relaxed hypothesis of Theorem~C.
%% \end{conjecture}

\subsection{Results in ergodic language}\label{SS:ergodic}
All along the article we shall use the term {\it measure preserving
  system}, or the word {\it system}, to designate a quadruple
$(X,\mathcal{B},\mu, T)$, where $(X,\mathcal{B},\mu)$ is a Lebesgue
probability space, and $T\colon X\to X$ is an \emph{invertible}
measurable map such that $\mu(T^{-1}A)=\mu(A)$ for every
$A\in\mathcal{B}$. The necessary background from ergodic theory is
given in Section~\ref{S:backgroundergodic}.

We shall use the following correspondence
principle of Furstenberg (the formulation given is from \cite{Be1}) to
reformulate Theorems~A, B, and C, in
ergodic theoretic language:
\begin{Correspondence1}\label{correspondence} Let $\Lambda$ be a
  set of integers.

  Then there exist a system $(X,\B,\mu,T)$ and a set
  $A\in\mathcal{B}$, with $\mu(A)=\bar{d}(\Lambda)$, and such that
  \begin{equation}\label{E:correspondence}
    \bar{d}(\Lambda \cap (\Lambda-n_1)\cap\ldots\cap
    (\Lambda-n_\ell))\geq \mu(A\cap T^{-n_1}A\cap\cdots \cap T^{-n_\ell}A),
  \end{equation}
  for every $n_1,\ldots,n_\ell\in\Z$ and $\ell\in\N$.
\end{Correspondence1}
 For convenience we give the following
definition:
\begin{definition}
  If $\ell\in\N$, we say that the set $S$ of integers is {\it a set of
    $\ell$-recurrence for the system $(X,\mathcal{B},\mu,T)$}, if for
  every $A\in \mathcal{B}$ with $\mu(A)>0$ we have
  \begin{equation}\label{E:one}
    \mu(A\cap T^{-s}A \cap T^{-2s}A\cap\dots \cap T^{-\ell s}A )>0
    \text{ for some non-zero } s\in S.
  \end{equation}
  We say that the set of integers $S$ is {\it a set of
    $\ell$-recurrence}, or {\it good for $\ell$-recurrence}, if it is
  a set of $\ell$-recurrence for every system. If $S$ is a set of
  $\ell$-recurrence for every $\ell\in\N$, we say that $S$ is \emph{a
    set of multiple recurrence}.
\end{definition}
\begin{remarks} $\bullet$ If $S$ is a set of
  $\ell$-recurrence, then \eqref{E:one} will be in fact satisfied for
  infinitely many $s\in S$ (in fact $S\cap m\Z$ is also a set of
  $\ell$-recurrence for every $m\in\N$).

  $\bullet$ We get a similar definition for sequences of integers by
  letting $S$ to be the range of the sequence. In this case we say that
  a sequence is \emph{good for $\ell$-recurrence}, or \emph{good for multiple recurrence}.
\end{remarks}
Let us give some examples of sets of multiple recurrence and also
mention some obstructions to recurrence. In the introduction we
mentioned that if $p\in \Z[x]$ is non-constant and $p(0)=0$, then the
sequence $p(n)$ is good for multiple recurrence (\cite{BL}).
%% More generally, if $p\in \R[x]$ is non-constant and $p(0)=0$ then
%% the sequence $[p(n)]$ is good for %%multiple recurrence (\cite{BH}).
Other examples of sets of multiple recurrence are IP sets,
meaning sets that consist of all finite sums (with distinct entries)
of some infinite set (\cite{FuK2}), and sets of the form
$\bigcup_{n\in\N}\{a_n,2a_n,\ldots,na_n\}$ where $a_n\in\N$ (follows
from a finite version of Szemer\'edi's theorem).

Examples of sets that are bad for single recurrence are sets that do
not contain multiples of some positive integer, and also the range of
lacunary sequences. It follows that the sequences $3n+2, n^2+1, p+2$
$(p \text{ prime})$, $n!$, are bad for single recurrence. The set
$\big\{n\in\N\colon \{n\sqrt{5}\}\in [1/2,3/4]\big\}$ and the sequence
$[\sqrt{5}n+2]$ are bad for recurrence for the rotation by $\sqrt{5}$
on $\T$ (on the other hand the sequences $[\sqrt{5}n+1]$ and $[\sqrt{5}n+3]$
are good for single
recurrence, see the discussion in Section~\ref{SS:single}).

Using Furstenberg's correspondence principle it is easy to see that
the following result implies Theorem~A (in fact it is not hard to show
that they are equivalent):
\begin{theoremA'}
  Let $a\in\H$  satisfy $x^k\prec a(x) \prec x^{k+1}$ for some
  non-negative integer $k$.

  Then $S=\{[a(1)],[a(2)],\ldots\}$ is a set of multiple recurrence.
\end{theoremA'}

The following result implies Theorem~B:
\begin{theoremB'}
Suppose that for every $r,m\in \N$ the set $S\subset \Z$
  contains patterns of the form $$\{r(c_{r,m}+mn)\colon 1\leq
  n\leq N_{r,m}\}$$ where $c_{r,m}, N_{r,m}$ are integers with $\lim_{m\to\infty}N_{r,m}=\infty$
  for every fixed $r\in \N$.

   Then $S$ is a set of multiple recurrence.
\end{theoremB'}
The following result implies Theorem~C:
\begin{theoremC'}
 Let $a\in\H$ have  polynomial growth and suppose that  $|a(x)-cp(x)|\to \infty $ for every $p\in
  \Z[x]$ and
  $c\in \R$.

  Then $S=\{[a(1)],[a(2)],\ldots\}$ is a set of single recurrence.
\end{theoremC'}

\subsection{Structure of the article}
In Section~\ref{S:backgroundergodic} we give the necessary background from ergodic
theory. Key to our study are some results about the structure of the
characteristic factors of some multiple ergodic averages.

In Section~\ref{S:backgroundnil} we give the necessary background on
nilsystems, and state some equidistribution results of sequences on
nilmanifolds. A crucial ingredient for our study is the quantitative equidistribution
result stated in Theorem~\ref{T:GT} for connected groups.  We
generalize this result to not necessarily connected groups in
Theorem~\ref{T:GT2}.

In Section~\ref{S:singlewarmup} we prove Theorem~B' which serves as a
model for the more complicated result that involves Hardy field sequences (Theorem~A').

In Section~\ref{S:polypattern} we carry out the first step needed to
prove Theorem~A'. We show that if $a\in\H$ satisfies $x^k\prec a(x)
\prec x^{k+1}$ for some non-negative integer $k$, then the range of the
sequence $[a(n)]$ some conveniently chosen polynomial patterns.

In Section~\ref{S:mainpart} we work with the patterns found in
Section~\ref{S:polypattern} and  carry out the final two steps of
the proof of Theorem~A'. The first step is a  ``reduction to nilsystems'' argument.
On the second step we verify a multiple recurrence property for nilsystems. Our argument is similar
to the one used to prove our model result Theorem~B'. The only extra difficulty occurs
in the ``reduction to nilsystems" step which happens to be technically much more
involved than the one needed for the our model result.
%% we need to establish an
%%extension of Proposition~\ref{P:seminorm'} that is technically much
%%more involved (Proposition~\ref{P:seminorm2}).

The Appendix  contains the proof of Theorem~C'.

\subsection{Further directions}\label{SS:conjectures}
%%\footnote{After the final version of this article was submitted, several of the problems mentioned on this section have %%been solved by the first author.
%%See arXiv:}
Roughly speaking, the method used to prove Theorem~A, or its equivalent version Theorem~A', amounts to finding
conveniently chosen polynomial pieces within the range of a sequence. These  pieces should be chosen
 so that it is possible to $(i)$ carry out
a reduction to nilsystems step, and $(ii)$ verify a certain recurrence property for nilsystems.
In view of the
 tools that have recently surfaced
 and help us
carry out steps $(i)$ and $(ii)$,
this ``polynomial method" appears to be rather flexible, and is very likely to find further applications.
For example, it  now  looks within reach  to show that every set of integers with positive density contains patterns of the form $m, m+[a_1(n)],\ldots,m+[a_k(n)]$ for ``most" choices of functions $a_i(x)$ that belong to some Hardy field and  have polynomial growth.
  This belief is reinforced by recent extensions in \cite{BH1} of the weakly mixing PET from \cite{Be2}.

 A more challenging problem is to find an example of a Hardy sequence of super-polynomial growth that is ``good'' for  Szemer\'edi's theorem (i.e. the conclusion of Theorem~A holds). For example, is  the sequence $[n^{\log\log{n}}]$
%% $[e^{(\log{n})^{1+a}}]$, where $a>0$ is small,
good for  Szemer\'edi's theorem?
%%for $a<1/3$ (\cite{Ba})).
  Concerning convergence results,  if $a\in\H$ satisfies the growth assumptions of Theorem~A, then it seems likely that the range of the sequence $[a(n)]$ can be split into ``polynomial pieces" that we can control, and hence  prove convergence in $L^2$ for the multiple ergodic averages
$$
\frac{1}{N}\sum_{n=1}^N T^{[a(n)]}f_1\cdot \ldots \cdot T^{\ell[a(n)]}f_\ell.
$$

Since the growth assumptions in Theorem~A can be relaxed when $\ell=1$ (see Theorem~C), it seems very likely that the same should be the case for general $\ell$:
\begin{conjectureA}\label{Con:A}
  Let $a\in \H$ have  polynomial growth and suppose that $|a(x)-cp(x)|\to\infty $ for every $p\in\Z[x]$ and $c\in\R$.

  Then for every $\ell\in\N$, every $\Lambda\subset\mathbb{Z}$ with
  $\bar{d}(\Lambda)>0$ contains arithmetic progressions of the form
  \begin{equation}\label{E:conA}
    \{m, m+[a(n)],m+2[a(n)],\ldots, m+\ell[a(n)]\}
  \end{equation}
  for some $m\in \Z $ and $n\in\N$ with $[a(n)]\neq 0$.
\end{conjectureA}
In view of the fact that Szemer\'edi's theorem on arithmetic progressions was a key ingredient in  showing that the primes contain arbitrarily long arithmetic progressions (\cite{GT1}), and likewise the polynomial Szemer\'edi theorem  was key in
  establishing polynomial progressions in the primes (\cite{TZ}), the following result seems plausible:
\begin{conjectureB}\label{Con:B}
Let $a\in \H$ satisfy the growth assumptions of Theorem~A $($or Conjecture~A$)$.

Then the  prime numbers contain arbitrarily long arithmetic progressions of the form \eqref{E:conA}.
\end{conjectureB}

\subsection{Notational conventions.} The following notation will be
used throughout the article: $\N=\{1,2,\ldots\}$, $Tf=f\circ T$,
$e(x)=e^{2\pi i x}$, $[x]$ denotes the integer part of $x$,  $\{x\}=x-[x]$, $\norm{x}=d(x,\Z)$,
$o_{m_1,\ldots,m_k}(1)$ denotes a quantity that goes to zero when
$m_1,\ldots,m_k\to +\infty$, by $a(x)\prec b(x)$ we mean
$\lim_{x\to\infty}a(x)/b(x)=0$, when there is no danger of confusion
we write $\infty$ instead of $+\infty$. We use the symbol $\ll$  when
 some expression is  majorized by a constant multiple of some  other expression.
We shall frequently abuse notation and denote the elements $t\Z$ of $\T$, where $t\in \R$,  by $t$.

\bigskip

\noindent {\bf Acknowledgment.} The authors would like to thank
Terence Tao for helpful discussions related to the material of
Section~\ref{S:backgroundnil},  the referee  for helpful comments,
and Pavel Zorin-Kranich for pointing out a mistake on the statement
of Theorems~B and B'.

\section{Background in ergodic theory}\label{S:backgroundergodic}
Below we gather some basic notions and facts from ergodic theory that
we use throughout the paper. The reader can find further background
material in ergodic theory in  \cite{Fu2}, \cite{Pe},
\cite{Wa}.
%% \subsection{Background from ergodic theory}
%% In this appendix ee gather several notions and background results
%% from ergodic theory.
\subsection{Factors in ergodic theory}
%% Throughout the article we consider {\it invertible} measure
%% preserving systems $(X,\mathcal{B},\mu, T)$ %%where the probability space $(X,\mathcal{B},\mu)$ is a {\it Lebesgue space}. This classical assumption %%allows us to use Rokhlin's theory of factors and disintegration.
%% (The basic reference here is \cite{Ro}, see also \cite{Zi} Section
%% 1.1, \cite{Rud} Chapter 2, or %%\cite{W} Section~2.3.) These two extra assumptions are not at all restrictive for our purposes, the %%reason being that the measure preserving systems constructed using the correspondence principle of %%Furstenberg are invertible and Lebesgue.
A {\it homomorphism} from a system $(X,\mathcal{B},\mu, T)$ onto a
system $(Y, \cD, \nu, S)$ is a measurable map $\pi\colon X'\to Y'$,
where $X'$ is a $T$-invariant subset of $X$ and $Y'$ is an
$S$-invariant subset of $Y$, both of full measure, such that
$\mu\circ\pi^{-1} = \nu$ and $S\circ\pi(x) = \pi\circ T(x)$ for $x\in
X'$. When we have such a homomorphism we say that the system $(Y, \cD,
\nu, S)$ is a {\it factor} of the system $(X,\mathcal{B},\mu, T)$.  If
the factor map $\pi\colon X'\to Y'$ can be chosen to be injective,
then we say that the systems $(X,\cB, \mu, T)$ and $(Y, \cD, \nu, S)$
are {\it isomorphic} (bijective maps on Lebesgue spaces have
measurable inverses).

%% A {\it factor} of a system $(X,\mathcal{B},\mu)$ is a (class modulo
%% isomorphism equivalence relation of) system $(Y, \cD, \nu, S)$ such
%% that there exists a homomorphism $\pi$ from the first system onto
%% the second one.

A factor can be characterized (modulo isomorphism) by the data
$\pi^{-1}(\mathcal{D})$ which is a $T$-invariant sub-$\sigma$-algebra
of $\mathcal B$, and any $T$-invariant sub-$\sigma$-algebra of
$\mathcal B$ defines a factor; by a classical abuse of terminology we
denote by the same letter the $\sigma$-algebra $\mathcal{D}$ and its
inverse image by $\pi$. In other words, if $(Y, \cD, \nu, S)$ is a
factor of $(X,\mathcal{B},\mu,T)$, we think of $\cD$ as a
sub-$\sigma$-algebra of $\mathcal{B}$. A factor can also be
characterized (modulo isomorphism) by a $T$-invariant subalgebra
$\mathcal{F}$ of $L^\infty(X,\mathcal{B},\mu)$, in which case $\cD$ is
the sub-$\sigma$-algebra generated by $\mathcal{F}$, or equivalently,
$L^2(X,\cD,\mu)$ is the closure of $\mathcal{F}$ in
$L^2(X,\mathcal{B},\mu)$. We shall sometimes abuse notation and use the sub-$\sigma$-algebra $\mathcal{D}$ in place of
 the subspace $L^2(X,\cD,\mu)$. For example, if we write that a function is orthogonal to the factor $\mathcal{D}$,
  we mean it  is orthogonal to the subspace $L^2(X,\cD,\mu)$.

If $\cD$ is a $T$-invariant sub-$\sigma$-algebra of $\cB$ and $f\in
L^2(\mu)$, we define the {\it conditional expectation
  $\mathbb{E}(f|\cD)$ of $f$ with respect to $\cD$} to be the
orthogonal projection of $f$ onto $L^2(\cD)$. We frequently make use
of the identities
$$
\int \mathbb{E}(f|\cD) \ d\mu= \int f\ d\mu, \quad
T\,\mathbb{E}(f|\cD)=\mathbb{E}(Tf|\cD).
$$
%% (If we want to indicate the dependence on the reference measure, we
%% write %%$\mathbb{E}=\mathbb{E}_\mu$.)

%% For each $d\in\mathbb{N}$, we define $\mathcal{K}_d$ to be the
%% factor induced by the function algebra
%%$$\{f\in L^\infty(\mu):T^df=f\}.$$
%%We define the {\it rational Kronecker factor} $\mathcal{K}_{rat}$ to
%% be the factor induced by the algebra generated by the functions
%%$$
%%\{f\in L^\infty(\mu):T^df=f \text{ for some } d\in\mathbb{N}\}\ .
%%$$
%%This algebra is the same as the algebra spanned by the bounded
%% functions that satisfy $Tf=e(a)\cdot f$ for some $a\in \Q$.

%% The {\it Kronecker factor} $\mathcal{K}$ is induced by the algebra
%% spanned by the bounded eigenfunctions of $T$, that means, functions
%% that satisfy $Tf=e(a)\cdot f$ for some $a\in \R$.

%% It is known that if $f$ is a bounded function such that
%% $\mathbb{E}_\mu(f|\mathcal{K}(T))=0$, then %%$\mathbb{E}_{\mu\otimes\mu}(f\otimes f|\mathcal{K}_{rat}(T\times T))=0$ (see for example \cite{Fu2}, %%Section 4.4).

The transformation $T$ is {\it ergodic} if $Tf=f$ implies that $f=c$
(a.e.) for some $c\in \mathbb{C}$, and \emph{totally ergodic} if
$T^rf=f$ for some $r\in\N$ implies that $f=c$ (a.e.) for some $c\in
\mathbb{C}$.

Every system $(X,\cB,\mu,T)$ has an {\it ergodic decomposition},
meaning that we can write $\mu=\int \mu_t\ d\lambda(t)$, where
$\lambda$ is a probability measure on $[0,1]$ and $\mu_t$ are
$T$-invariant probability measures on $(X,\cB)$ such that the systems
$(X,\cB,\mu_t,T)$ are ergodic for $t\in [0,1]$. We sometimes denote
the ergodic components by $T_t, t\in [0,1]$.

We say that $(X,\cB,\mu,T)$ is an {\it inverse limit of a sequence of
  factors} $(X,\cB_j,\mu,T)$ if $(\cB_j)_{j\in\mathbb{N}}$ is an
increasing sequence of $T$-invariant sub-$\sigma$-algebras such that
$\bigvee_{j\in\N}\mathcal{B}_j=\mathcal{B}$ up to sets of measure
zero.

\subsection{Characteristic factors for polynomial averages}
Following \cite{HK1}, for every system $(X,\mathcal{B},\mu,T)$ and
function $f\in L^\infty(\mu)$, we define inductively the
 seminorms $\nnorm{f}_\ell$ as follows: For $\ell=1$ we set
$\nnorm{f}_1=\int |\E(f|\mathcal{I})|d\mu$, where $\mathcal{I}$ is the
$\sigma$-algebra of $T$-invariant sets. For $\ell\geq 2$ we set
\begin{equation}
  \label{eq:recur} \nnorm f_{\ell+1}^{2^{\ell+1}} =\lim_{N\to\infty}\frac
  1N\sum_{n=1}^N \nnorm{\bar{f}\cdot T^nf}_{\ell}^{2^{\ell}}.
\end{equation}
It was shown in~\cite{HK1} that for every integer
$\ell\geq 1$, $\nnorm\cdot_\ell$ is a seminorm on $L^\infty(\mu)$ and
it defines factors $\cZ_{\ell-1}=\cZ_{\ell-1}(T)$ in the following
manner: the $T$-invariant sub-$\sigma$-algebra $\cZ_{\ell-1}$ is
characterized by
%% for $f\in L^\infty(\mu)$ define the $\sigma$-algebra $\cZ_{k-1}(X)$
%% of $\cX$ by
$$
\text{ for } f\in L^\infty(\mu),\ \E(f|\cZ_{\ell-1})=0\text{ if and
  only if } \nnorm f_{\ell} = 0.~\footnote{In \cite{HK1} the authors work
  with ergodic systems, in which case $\nnorm{f}_1=\left|\int f \
    d\mu\right|$, and real valued functions, but the whole discussion
  can be carried out for non-ergodic systems as well and complex
  valued functions without extra difficulties.}
$$
We call $\cZ_\ell$ the \emph{$\ell$-step nilfactor} of the system.
 By $\cZ$ we denote the smallest factor that is an extension of all the factors
$\cZ_\ell$ for $\ell\in\N$, and we call $\cZ$ the \emph{nilfactor} of
the system.  If $f$ is a bounded function that satisfies
$\E_\mu(f|\cZ_\ell(T))=0$, then $\E_{\mu\otimes\mu}(f\otimes
\overline{f}|\cZ_{\ell-1}(T\times T))=0$ (this is implicit in \cite{HK1}). Also,
if $T_t$ where $t\in [0,1]$ are the ergodic components of the system,
then $\E(f|\mathcal{Z}_{\ell}(T))=0$ if and only if
$\E(f|\mathcal{Z}_\ell(T_t))=0$ for a.e. $t\in[0,1]$.

%% We note that for ergodic systems the factor $\cZ_0=\mathcal{I}$ is
%% trivial and $\cZ_1=\mathcal{K}$.
The factors $\cZ_\ell$ are of particular interest because they control
the limiting behavior in $L^2$ of several multiple
ergodic averages. The next result makes this more precise.
\begin{theorem}[{\bf Leibman~\cite{L3}}]\label{T:L2}
  Let $p_1,\ldots,p_s\colon \Z^r\to \Z$ be a family of non-constant essentially
  distinct polynomials.

Then   there exists a non-negative integer $\ell=\ell(p_1,p_2,\ldots,p_s)$
  with the following property: If $(X,\cB,\mu,T)$ is a system and at least
  one of the functions
  $f_1,\ldots,f_s\in L^\infty(X)$ is orthogonal to the factor $\mathcal{Z}_\ell(T)$,
   then for every F{\o}lner sequence
  $(\Phi_N)_{N\in\N}$ in $\Z^r$ we have
$$
\lim_{N\to\infty}\norm{\frac{1}{|\Phi_N|}\sum_{n\in \Phi_N}
T^{p_1(n)}f_1\cdot T^{p_2(n)}f_2\cdot\ldots\cdot T^{p_s(n)}f_s}_{L^2(\mu)}=0.
$$
\end{theorem}
We say that $\mathcal{Z}_\ell(T)$ is a {\it characteristic factor}
associated to the family $p_1,p_2,\ldots,p_s$ when this last fact is true.

We shall also use the following easy corollary of the previous result:
\begin{corollary}\label{C:L2'}
 Let $(X,\cB,\mu,T)$ be a system,  $p_1,\ldots,p_s\colon \Z^r\to \Z$ be a family of
 non-constant essentially distinct polynomials,
  and
  let $\mathcal{Z}_\ell(T)$ be a characteristic factor for this
  family.  Suppose that at least one of the functions $f_0,f_1,\ldots,f_s\in
  L^\infty(X)$  is orthogonal to the
  factor $\mathcal{Z}_{\ell+1}(T)$.

   Then for every F{\o}lner sequence
  $(\Phi_N)_{N\in\N}$ in $\Z^r$ we have
$$
\lim_{N\to\infty}\frac{1}{|\Phi_N|}\sum_{n\in \Phi_N}\left|\int
  f_0\cdot T^{p_1(n)}f_1\cdot\ldots\cdot T^{p_s(n)}f_s \
  d\mu\right|^2=0.
$$
\end{corollary}
\begin{proof}
  If $f_i$ is orthogonal to the factor $\mathcal{Z}_{\ell+1}(T)$, then
as mentioned above, the function   $f_i\otimes \overline{f}_i$ is orthogonal to the factor
  $\mathcal{Z}_{\ell}(T\times T)$. Therefore, by Theorem~\ref{T:L2} the averages
$$
\frac{1}{|\Phi_N|}\sum_{n\in \Phi_N}\int \int f_0(x) \cdot
\bar{f_0}(y)\cdot f_1(T^{p_1(n)}x)\cdot \bar{f_1}(T^{p_1(n)}y)\cdot
\ldots \cdot f_s(T^{p_s(n)}x)\cdot \bar{f_s}(T^{p_s(n)}y)\
d\mu(x)d\mu(y)
$$
converge to zero. This immediately implies the advertised result.
\end{proof}
Host and Kra (\cite{HK1}) showed that the factors $\mathcal{Z}_\ell$
are of purely algebraic structure (a closely related result was
subsequently proved by Ziegler (\cite{Z})), a result that is crucial
for our study.
\begin{theorem}[{\bf Host \& Kra~\cite{HK1}}]\label{T:HoKra}
  Let $(X,\mathcal{B},\mu,T)$ be a system and $\ell\in \N$.

  Then  a.e. ergodic
   component of the  factor $\mathcal{Z}_\ell(T)$ is an inverse limit of $\ell$-step
  nilsystems.
\end{theorem}
This result  justifies our name for the factors $\mathcal{Z}_\ell(T)$.
%%Because of this result we call $\cZ_\ell$ the
%%\emph{$\ell$-step nilfactor} of the system. The smallest factor that is an extension of
%%all finite step nilfactors is denoted by $\cZ$ and called the \emph{nilfactor} of
%%the system.

\section{Equidistribution results on
  nilmanifolds}\label{S:backgroundnil}
%% \subsection{Background from ergodic theory}
In this section we give some background material on nilsystems and
gather some equidistribution results of polynomial sequences on
nilmanifolds that will be used later.  Nilsystems play a central role
in our study because they provide a sufficient class for verifying
several multiple recurrence results for general measure preserving
systems. In fact, when one deals with ``polynomial recurrence"
  this is usually a consequence of Theorems~\ref{T:L2} and
  \ref{T:HoKra}. These two results, taken together, show that nilsystems
  control the limiting behavior of the corresponding  polynomial
  multiple ergodic averages.
\subsection{Nilmanifolds, definition and basic
  properties}\label{SS:basic}
The reader can find fundamental properties of nilsystems related to
our study in \cite{AGH}, \cite{Pa}, \cite{Les}, and \cite{L2}.
%% Below we summarize some facts that we shall use, all the proofs can
%% be found in \cite{L2}.

Given a topological group $G$, we denote the identity element by $e$,
and we let $G_0$ denote the connected component of $e$.  If $A,
B\subset G$, then $[A,B]$ is defined to be the subgroup generated by
elements of the form $\{[a,b]:a\in A, b\in B\}$ where $[a,b]=ab
a^{-1}b^{-1}$. We define the commutator subgroups recursively by
$G_1=G$ and $G_{k+1}=[G, G_{k}]$. A group $G$ is said to be {\it
  $k$-step nilpotent} if its $(k+1)$ commutator $G_{k+1}$ is trivial.
If $G$ is a $k$-step nilpotent Lie group and $\gG$ is a discrete
cocompact subgroup, then the compact space $X = G/\gG$ is said to be a
{\it $k$-step nilmanifold}.  The group $G$ acts on $G/\gG$ by left
translation where the translation by a fixed element $a\in G$ is given
by $T_{a}(g\gG) = (ag) \gG$. By $m_X$ we  denote the unique probability
measure on $X$ that is invariant under the action of $G$ by left
translations (called the {\it normalized Haar measure}) and $\G/\gG$ denote the
Borel $\sigma$-algebra of $G/\gG$. Fixing an element $a\in G$, we call
the system $(G/\gG, \G/\gG, m, T_{a})$ a {\it $k$-step nilsystem}. We
call the elements of $G$ \emph{nilrotations}.

Given a nilmanifold $X=G/\Gamma$, an {\it ergodic
  nilrotation} is an element $a\in G$ such that the sequence
$(a^n\Gamma)_{n\in\N}$ is uniformly distributed on $X$. If $X$ is a connected nilmanifold and
$a\in G$ is an ergodic
nilrotation it can be shown that  for every $d\in \N$ the nilrotation
$a^d$ is also ergodic.

\begin{example}\label{E:1}
  On the space $G=\Z\times\R^2$,
  define multiplication as follows: \\
  if $g_1=(m,x_1,x_2)$ and $g_2=(n,y_1,y_2)$, then
$$
g_1\cdot g_2=(m+n,x_1+y_1, x_2+y_2+my_1).
$$
Then $G$ with $\cdot$ is a $2$-step nilpotent Lie group and the group
$G_0=\{0\}\times \R^2$ is Abelian. The discrete subgroup $\Gamma=\Z^3$
is cocompact and $X=G/\Gamma$ is connected. It can be shown that the
nilrotation $a=(1,\alpha,\beta)$ is ergodic if and only if $\alpha$
is an irrational number.
\end{example}

We remark that the representation of a nilmanifold $X$ as a homogeneous
space of a nilpotent Lie group $G$ is not unique. If $X$ is a
connected nilmanifold, it can be shown (\cite{L2}) that it admits a
representation of the form $X=G/\Gamma$ such that:
%%\begin{equation}\label{E:assumptions}
$G_0$ is simply connected and  $G=G_0\Gamma$.
%%\end{equation}
%%\centerline{{\it $G$ is simply connected (but not necessarily
%%    connected) and satisfies $G=G_0\Gamma$.}}
In the sequel, whenever
$X$ is connected, {\it we will always assume that $G$ satisfies these two
extra assumptions}.

\subsection{Qualitative equidistribution results on nilmanifolds}
If $G$ is a nilpotent group, then a sequence $g\colon \Z\to G$ of the
form $g(n)=a_1^{p_1(n)}\cdot\ldots\cdot a_k^{p_k(n)}$ where $a_i\in G$
and $p_i$ are polynomials taking integer values at the integers is
called a \emph{polynomial sequence in} $G$. If the maximum of the
degrees of the polynomials $p_i$ is at most $d$ we say that the
\emph{degree} of $g(n)$ is at most $d$.
%%\footnote{One should not
%%  confuse this notion of degree with the one defined in \cite{L1}
%%  which is completely different.}
A \emph{polynomial sequence on the
  nilmanifold} $X=G/\Gamma$ is a sequence of the form
$(g(n)\Gamma)_{n\in\Z}$ where $g\colon \Z\to G$ is a polynomial
sequence in $G$.

\begin{theorem}[{\bf Leibman~\cite{L2}}]\label{T:L}
  Suppose that $X = G/\gG$ is a connected nilmanifold and $g(n)$ is a
  polynomial sequence in $G$. Let $Z=G/([G_0,G_0]\Gamma)$ and
  $\pi\colon X\to Z$ be the natural projection.

  Then  for every   $x\in X$ the
  sequence $(g(n)x)_{n\in\mathbb{N}}$ is equidistributed in $X$ if and
  only if the sequence  $(g(n)\pi(x))_{n\in\mathbb{N}}$ is equidistributed in $Z$.
\end{theorem}
Note that $[G_0,G_0]$ is a normal subgroup of $G$ and so $G/[G_0,G_0]$
is a group.

\subsection{Quantitative equidistribution results on nilmanifolds}
\subsubsection{The case of a connected group} We will later use a
quantitative version of Theorem~\ref{T:L} that was recently obtained
by Green and Tao in \cite{GT}. In order to state it we need to review
some notions that were introduced in \cite{GT}.

Given a nilmanifold $X=G/\Gamma$, the {\em horizontal torus} is
defined to be the compact Abelian group $H=G/([G,G]\Gamma)$.  If $X$
is connected, then $H$ is isomorphic to some finite dimensional torus
$\T^l$. By $\pi\colon X\to H$ we denote the natural projection map.  A
\emph{horizontal character} is a continuous homomorphism $\chi$ of $G$ that
satisfies $\chi(g\gamma)=\chi(g)$ for every $\gamma\in\Gamma$. Since
every character annihilates $[G,G]$, every horizontal character factors
through $H$, and so can be thought of as a character of the horizontal
torus. Since $H$ is identifiable with a finite dimensional torus
$\T^l$ (we assume that $X$ is connected), $\chi$ can also be thought
of as a character of $\T^l$, in which case there exists a unique
$\kappa\in\Z^l$ such that $\chi(t)=\kappa\cdot t$, where $\cdot$
denotes the inner product operation.  We refer to $\kappa$ as the
frequency of $\chi$ and $\norm{\chi}=|\kappa|$ as the \emph{frequency
  magnitude} of $\chi$.

\begin{example}
  Let $X$ be as in Example $1$.  The map $\chi(m,x_1,x_2)=e(lx_1)$,
  where $l\in\Z$, is a horizontal character of $G$ and the map
  $\phi(m,x_1,x_2)=x_1 \pmod{1}$ induces an identification of the
  horizontal torus with $\T$.  Under this identification, $\chi$ is
  mapped to the character $\chi_1(t_1)=e(l_1t_1)$ of $\T$.
\end{example}

If $p\colon\Z\to \R$ is a polynomial sequence of degree $k$,
then $p$ can be uniquely expressed in the form $
p(n)=\sum_{i=0}^k\binom{n}{i}\alpha_i $ where $\alpha_i\in\R$. For $N\in\N$ we
define
\begin{equation}\label{E:norms}
\norm{e(p(n))}_{C^\infty[N]}=\max_{1\leq i\leq k}( N^i \norm{\alpha_i})
\end{equation}
where $\norm{x}=d(x,\Z)$.  One also gets a similar definition for polynomials
$p\colon \Z\to \T$.

%% A Malcev basis $\mathcal{X}$ (to be defined?) is $Q$-{\emph
%%   rational} if all the constants $c_{ijk}$ in %%the relations
%%$$
%%[X_i,X_j]=\sum_k c_{ijk}X_k
%%$$
%%are rational with numerator and denominator (in reduced form) less
%% than $Q$ in absolute value.

Given $N\in\N$, a finite sequence $(g(n)\Gamma)_{1\leq n\leq N}$ is
said to be $\delta$-\emph{equidistributed} if
$$
\Big|\frac{1}{N}\sum_{n=1}^N F(g(n)\Gamma)-\int_{X}F \ dm_X\Big|\leq
\delta \norm{F}_{\text{Lip}}
$$
for every Lipschitz function $F\colon X\to \C$ where
$$
\norm{F}_{\text{Lip}}=\norm{F}_\infty+ \sup_{x,y\in X, x\neq
  y}\frac{|F(x)-F(y)|}{d_X(x,y)}
$$
for some appropriate metric $d_X$ on $X$.\footnote{The metric $d_X$ is
  defined in \cite{GT} using a Malcev basis $\mathcal{X}$ of $X$, so
  the notion of equidistribution we get does depend on the choice of
  the Malcev basis $\mathcal{X}$. The exact definition of $d_X$ will
  not be needed anywhere in our article, so we omit it.}  We can now
state the quantitative equidistribution result that we shall use. It  can be easily
derived from \cite{GT}.
\begin{theorem}[{\bf Green \& Tao~\cite{GT}}]\label{T:GT}
  Let $X=G/\Gamma$ be a nilmanifold with $G$ connected and simply
  connected and $d\in\N$.\footnote{In our context, we assume that the
    Malcev basis and hence the metric on $X$ is fixed. So unlike
    the more refined result stated in \cite{GT}, for the result we state here
     there is no reason to impose
    restrictions on the Malcev basis we use (or even refer to it).}

  Then there exists $C=C_{X,d}>0$ with the following property: For every
  $N\in\N$ and  $\delta$ small enough,
  %% Suppose that $X=G/\Gamma$ is a $d$-step $m$-dimensional
  %% nilmanifold with a $\delta^{-1}$-rational Malcev basis and
  %% suppose that $G$ is connected.
  if $g\colon \Z\to G$ is a polynomial sequence of degree at most $d$
  such that the finite sequence $(g(n)\Gamma)_{1\leq n\leq N}$ is not
  $\delta$-equidistributed, then there exists a non-trivial horizontal character
  $\chi$, with frequency magnitude $\norm{\chi}\leq \delta^{-C}$, such
  that
  \begin{equation}\label{E:badequi}
    \norm{\chi( g(n))}_{C^\infty[N]}\leq  c_1\delta^{-C}
  \end{equation}
%%$$
%%\norm{\chi\circ g}_{C^\infty[N]}\leq c_1 \delta^{-C}
%%$$
  for some absolute constant $c_1$, where $\chi$ is thought of as a
  character of the horizontal torus $H=\T^l$ and $g(n)$ in $\eqref{E:badequi}$  as a
  polynomial sequence in $\T^l$.
\end{theorem}
\begin{remarks}
  $\bullet$ We are not going to make use of the explicit form of the
  upper bounds on $\norm{\chi}$ and $\norm{\chi(
    g(n))}_{C^\infty[N]}$, any upper bound that depends only on
  $\delta,X,$ and $d$, will do just fine.

  $\bullet$ Condition \eqref{E:badequi} implies that the finite
  sequence $(\pi(g(n)\Gamma))_{1\leq n<N_1}$ is not
  $(c_2\delta^C)$-equidistributed in $G/([G,G]\Gamma)$ for all every
  $N_1<c_2\delta^CN$, for some absolute constant $c_2$.
\end{remarks}

\begin{example}
  It is instructive to interpret the previous result in some special
  case. Let $X=\T$ (with the standard metric), and suppose that
  the polynomial
  sequence on $\T$ is given by $p(n)=n^d\alpha+q(n)$ where $d\in \N$,
  $\alpha\in \R$, and $q\in\Z[x]$ satisfies $\deg{q}\leq d-1$. In this case
  Theorem~\ref{T:GT} reads as follows: There exists $C=C_d>0$ such
  that for every $N\in\N$ and every  $\delta$ small enough,
  %% Suppose that $X=G/\Gamma$ is a $d$-step $m$-dimensional
  %% nilmanifold with a $\delta^{-1}$-rational Malcev basis and
  %% suppose that $G$ is connected.
  if the finite sequence $(n^d\alpha+q(n))_{1\leq n\leq N}$ is not
  $\delta$-equidistributed in $\T$, then $\norm{k\alpha}\leq
  c_1\delta^{-C}/N^d$ for some $k\in \Z$ with $|k|\leq \delta^{-C}$
  and some absolute constant $c_1>0$.
\end{example}

%% The following result is an immediate consequence of a quantitative
%% uniform distribution result that was recently proved in \cite{GT}:
%% \begin{theorem} [{\bf Green \& Tao~\cite{GT}}]\label{T:GT}
%%   Let $X=G/\Gamma$ be a nilmanifold with $G$ connected and
%%   $d\in\N$.  Then for every $\varepsilon>0$ there exists
%%   $\delta=\delta_{\varepsilon,X,d}>0$ with the following %%property: for every $N\in\N$ and polynomial sequence
%%   $g(n)$ of degree $d$ if $(g(n))_{1\leq n\leq N}$ is
%%   $\delta$-equidistributed on $T=G/([G,G]\Gamma)$ then
%%   $(g(n))_{1\leq n \leq N}$ is $\varepsilon$-equidistributed in
%%   $X$.
%% \end{theorem}
\subsubsection{The general case} \label{S:GT2} In this subsection we
establish an extension of Theorem~\ref{T:GT} to the case where the
group $G$ is not necessarily connected (but we always assume that
$X=G/\Gamma$ is connected and $G$ is simply connected).

Let $G$ be a group.  A map $T\colon G\to G$ is said to be
\emph{affine} if $T(g) = b\cdot S(g)$ for some homomorphism $S$ of $G$
and $b\in G$.  The homomorphism $S$ is said to be {\em unipotent} if
there exists $n\in\N$ so that $(S-{\text Id})^{n}=0$.  In this case we
say that the affine transformation $T$ is a unipotent affine
transformation.

If $X=G/\Gamma$ is a connected nilmanifold, the {\em affine torus} of
$X$ is defined to be the homogeneous space $A=G/([G_0,G_0]\Gamma)$.
The next lemma (whose statement and proof are reproduced from
\cite{FK}) explains our terminology (notice that if $H$ is the group
$G/[G_0,G_0]$, then $H_0$ is Abelian).
\begin{proposition}[{\bf F. \& Kra}~\cite{FK}]\label{P:FK}
  Let $X=G/\Gamma$ be a connected nilmanifold and suppose that the group $G_0$ is
  Abelian.

  Then the nilrotations $T_a(x)=ax$, $a\in G$, defined on $X$
  with the normalized Haar measure $m_X$, are simultaneously
  isomorphic to a collection of unipotent affine transformations on
  some finite dimensional torus with the normalized Haar measure. Furthermore,
  the conjugation map can be taken to be continuous.
\end{proposition}
\begin{proof}
  We start with a reduction. As we mentioned in Section~\ref{SS:basic}
  since $X$ is connected we can assume that $G=G_0\Gamma$.  We claim
  that under our additional assumption that $G_0$ is Abelian we have
  that $\Gamma_0=\Gamma\cap G_0$ is a normal subgroup of $G$. Let
  $\gamma_0\in \Gamma_0$ and $g=g_0\gamma$, where $g_0\in G_0$ and
  $\gamma\in \Gamma$. Since $G_0$ is normal in $G$, we have that
  $g^{-1}\gamma_0 g\in G_0$. Moreover,
$$
g^{-1} \gamma_0 g=\gamma^{-1} g_0^{-1} \gamma_0 g_0 \gamma=
\gamma^{-1} \gamma_0 \gamma \in\Gamma,
$$
the last equality being valid since $G_0$ is Abelian. Hence,
$g^{-1}\gamma_0 g\in\Gamma_0$ and $\Gamma_0$ is normal in $G$, proving our claim.
After substituting $G/\Gamma_0$ for $G$ and
$\Gamma/\Gamma_0$ for $\Gamma$, we have
$X=(G/\Gamma_0)/(\Gamma/\Gamma_0)$. Therefore, we can assume that $G_0\cap
\Gamma=\{e\}$. Note that now $G_0$ is a connected compact Abelian Lie
group, and so is isomorphic to some finite dimensional torus $\T^d$.

Every $g\in G$ is uniquely representable in the form $g=g_0\gamma$,
with $g_0\in G_0$, $\gamma \in \Gamma$. The map $\phi\colon X\to G_0$,
given by $\phi(g\Gamma)=g_0$ is a well defined homeomorphism. Since
$\phi(hg\Gamma)=h\phi(g\Gamma)$ for every $h\in G_0$, the measure
$\phi(\mu)$ on $G_0$ is invariant under left translations.  Thus
$\phi(m)$ is the normalized Haar measure on $G_0$. If $a=a_0\gamma$,
$g=g_0\gamma'$ with $a_0,g_0\in G_0$ and $\gamma,\gamma'\in\Gamma$,
then $ag\Gamma=a_0\gamma g_0 \gamma^{-1}\Gamma$.  Since $\gamma g_0
\gamma^{-1} \in G_0$, we have that $\phi(ag\Gamma)=a_0\gamma
g_0\gamma^{-1}$. Hence $\phi$ conjugates $T_a$ to $T_a'\colon G_0\to
G_0$ defined by
$$
T_a'(g_0)=\phi T_a \phi^{-1}=a_0\gamma g_0 \gamma^{-1}.
$$
Since $G_0$ is Abelian this is an affine map; its linear part
$g_0\mapsto \gamma g_0 \gamma^{-1}$ is unipotent since $G$ is
nilpotent.  Letting $\psi\colon G_0\to \T^d$ denote the isomorphism
between $G_0$ and $\T^d$, we have that $T_a$ is isomorphic to the
unipotent affine transformation $S=\psi T_a'\psi^{-1}$ acting on
$\T^d$.
\end{proof}
Because of this lemma, we can identify the affine torus $A$ of a
nilmanifold $X$ with a finite dimensional torus $\T^l$ and think of a
nilrotation acting on $A$ as a unipotent affine transformation on
$\T^l$.
\begin{example}Let $X$ be as in Example $1$.  We have $X\simeq
  (\Z\times \R^2)/(\Z\times {\bf 0})$, so we can assume that we have
  equality. If $a=(m,\alpha_1,\alpha_2)\in \Z\times \T^2$, then the map
  $\phi\colon \Z\times \T^2\to \T^2$, defined by
  $\phi(k,t_1,t_2)=(t_1,t_2) \pmod{1}$, factors through $X$, and
  conjugates the nilrotation $T_a(x)=ax$ to the unipotent affine
  transformation $S \colon \T^2\to\T^2$ defined by
$$
S(t_1,t_2)=(t_1+\alpha_1,t_2+mt_1+\alpha_2).
$$
\end{example}
A \emph{quasi-character} of a nilmanifold $X=G/\Gamma$ is a
function $\psi\colon G\to \C$
that is a continuous homomorphism of $G_0$ and satisfies $\psi(g\gamma)=\psi(g)$
for every $\gamma\in\Gamma$. Every quasi-character annihilates
$[G_0,G_0]$, so it factors through the affine torus $A$ of $X$. Under
the identification of Proposition~\ref{P:FK} we have that $A\simeq
\T^l$ and every quasi-character of $X$ is mapped to a character of
$\T^l$. Therefore,  thinking of $\psi$ as a character of $\T^l$ we have
$\psi(t)=\kappa\cdot t$ for some  $\kappa\in\Z^l$, where $\cdot$
denotes the inner product operation.  We refer to $\kappa$ as the
\emph{frequency} of $\psi$ and $\norm{\psi}=|\kappa|$ as the
\emph{frequency magnitude} of $\psi$.
\begin{example}
  Let $X$ be as in Example $1$.  The map
  $\psi(m,x_1,x_2)=e(l_1x_1+l_2x_2)$, where $l_1,l_2\in\Z$, is a
  quasi-character of $X$. Notice that $\psi$ is not a homomorphism of $G$ and
  and so it  is  not a  character of
  $X$.  The map $\phi(m,x_1,x_2)=(x_1,x_2) \pmod{1}$ induces an
  identification of the affine torus (in this case $A=X$) with $\T^2$.
  Under this identification, $\psi$ is mapped to the character
  $\psi_1(t_1,t_2)=e(l_1t_1+l_2t_2)$ of $\T^2$.
\end{example}
We are now ready to state the advertised extension of
Theorem~\ref{T:GT}:
\begin{theorem}[{\bf Corollary of
 %%of Green \&    Tao~
    Theorem~\ref{T:GT}}]\label{T:GT2}
  Let $X=G/\Gamma$ be a connected nilmanifold $($we  always assume that $G_0$ is simply connected$)$ and $d\in\N$.

  Then there
  exists $C=C_{X,d}>0$ with the following property: For every $N\in\N$
  and $\delta$ small enough,
  %% Suppose that $X=G/\Gamma$ is a $d$-step $m$-dimensional
  %% nilmanifold with a $\delta^{-1}$-rational Malcev basis and
  %% suppose that $G$ is connected.
  if $g\colon \Z\to G$ is a polynomial sequence of degree at most $d$
  such that the finite sequence $(g(n)\Gamma)_{1\leq n\leq N}$ is not
  $\delta$-equidistributed, then there exists a non-trivial quasi-character $\psi$
  with frequency magnitude $\norm{\psi}\leq \delta^{-C}$ such that
  \begin{equation}\label{E:quasi}
    \norm{\psi(g(n))}_{C^\infty[N]}\leq  c_1 \delta^{-C}
  \end{equation}
%%$$
%%\norm{\chi\circ g}_{C^\infty[N]}\leq c_1 \delta^{-C}
%%$$
  for some absolute constant $c_1$, where we think of $\psi$ as a
  character of some finite dimensional torus $\T^l$ $($the affine
  torus$)$ and $g(n)$ as a polynomial sequence of unipotent affine
  transformations on $\T^l$.
\end{theorem}
\begin{remark}
  We have $\psi(g(n))=e(p(n))$ for some $p\in \R[x]$ and so
  $\norm{\psi( g(n))}_{C^\infty[N]}$ is well defined.
\end{remark}

We first make some observations that will help us deduce
Theorem~\ref{T:GT2} from Theorem~\ref{T:GT}.  As we remarked in
Section~\ref{SS:basic}, if $X=G/\Gamma$ is a connected nilmanifold we
can assume that every $g\in G$ is representable in the form
$g_0\gamma$, where $g_0\in G_0$ and $\gamma\in \Gamma$. Therefore, $X=(G_0\Gamma)/\Gamma$ can be identified with the nilmanifold
$G_0/(G_0\cap \Gamma)$.
%% where $\Gamma_0=\Gamma\cap G_0$.
 If $a\in G$ we have
$a=a_0\gamma$ for some $a_0\in G_0$ and $\gamma\in \Gamma$. Since
$G_0$ is a normal subgroup of $G$ we have that $a^n=a_n\gamma^n$ for
some $a_n\in G_0$. Using this, one easily verifies that any degree $d$
polynomial sequence $g(n)$ in $G$ factors as follows:
$g(n)=g_0(n)\gamma(n)$ where $g_0(n)\in G_0$ for $n\in\N$ and
$\gamma(n)$ is a degree $d$ polynomial sequence in $\Gamma$. By
Proposition $3.9$ in \cite{L1} (for a more direct proof see
Proposition 4.1 in \cite{BLL}) we get that $g_0(n)$ is also a
polynomial sequence in $G_0$. Moreover, if $G$ is $k$-step nilpotent,
a close examination of the proof of Proposition 4.1 in \cite{BLL}
reveals that the degree of $g_0(n)$ is at most $dk$.
\begin{example}
  Let $X$ be the nilmanifold of Example $1$. We have $G_0=\{0\}\times
  \R^2$ and the map $\phi\colon \Z\times \T^2\to \T^2$, defined by
  $\phi(k,t_1,t_2)=(0,t_1,t_2)$, induces an identification between $X$
  and the nilmanifold $G_0/(G_0\cap \Gamma) \simeq \T^2$. For
  $a=(2,\alpha,\alpha)$ the polynomial sequence $g(n)=a^n$ in $G$
  factors as
$$
g(n)=(2n, n\alpha,n^2\alpha)=a_0^n \cdot b_0^{n^2}\cdot \gamma^n
$$
where $a_0=(0,\alpha,0), b_0=(0,0,\alpha)\in G_0$, and
$\gamma=(2,0,0)\in \Gamma$.  In this case we have that $g_0(n)=a_0^n
\cdot b_0^{n^2}$ is a degree $2$ polynomial sequence in $G_0$.
\end{example}
\begin{proof}[Proof of Theorem~\ref{T:GT2}]
  Let $C=C_{G_0/(G_0\cap \Gamma),kd}$ be the positive number defined in
  Theorem~\ref{T:GT}.  Suppose that $(g(n)\Gamma)_{1\leq n\leq N}$ is
  not $\delta$-equidistributed in $X=G/\Gamma$ for some small enough
  $\delta$.  As discussed before, we have
  $g(n)\Gamma=g_0(n)\Gamma$ where $g_0(n)$ is a polynomial sequence in
  $G_0$ of degree at most $kd$.  Since $X$ can be identified with
  $G_0/(G_0\cap \Gamma)$, it follows that the finite sequence
  $(g_0(n)(G_0\cap \Gamma))_{1\leq n\leq N}$ is not $\delta$-equidistributed
  in $G_0/(G_0\cap \Gamma)$.  Since $G_0$ is connected and simply connected,
  and the polynomial
  sequence $g_0(n)$ is defined in $G_0$, and has degree at most $kd$,
  Theorem~\ref{T:GT} applies. We get that there exists a non-trivial horizontal
  character $\chi_0$ of $G_0/(G_0\cap \Gamma)$ such that $\norm{\chi_0}\leq
  \delta^{-C}$ and
$$
\norm{\chi_0( g_0(n))}_{C^\infty[N]}\leq c_1 \delta^{-C}.
$$

We can lift $\chi_0$ to a quasi-character of $X$ as follows:
Consider the  discrete group $G/G_0$. Since $G=G_0\Gamma$  we have $G/G_0=\{\gamma G_0\colon \gamma\in \Gamma\}$.
Let  $\tilde{\Gamma}$  be a subset of $\Gamma$ so that the map $\gamma\to \gamma G_0$, from $\tilde{\Gamma}$ to $G/G_0$,
is bijective.
 Then every element $h\in G$ has a unique representation $h=h_0\tilde{\gamma}$
 with $h_0\in G_0$ and $\tilde{\gamma}\in \tilde{\Gamma}$.
We define the map $\psi\colon G\to \C$ by  $\psi(h)=\chi_0(h_0)$.
Since $\chi_0(g_0\gamma_0)=\chi_0(g_0)$ for every  $g_0\in G_0$ and $\gamma_0\in G_0\cap \Gamma$, it
 follows that
$\psi$  agrees with $\chi_0$ on  $G_0$.
Furthermore, writing $\gamma\in \Gamma$ as $\gamma=\gamma_0\tilde{\gamma}$ with $\gamma_0\in G_0\cap \Gamma$ and
$\tilde{\gamma}\in \tilde{\Gamma}$, and using again that $\chi_0(g_0\gamma_0)=\chi_0(g_0)$ for $g_0\in G_0$, one gets
that $\psi(g_0\gamma)=\psi(g_0\gamma_0\tilde{\gamma})=\chi_0(g_0\gamma_0)=\chi_0(g_0)=\psi(g_0)$ for every $g_0\in G_0$ and $\gamma\in \Gamma$.  Since every $g\in G$ can be written as
 $g=g_0\tilde{\gamma}$ for some  $g_0\in G_0$ and $\tilde{\gamma}\in \tilde{\Gamma}$, we conclude that
 $\psi(g\gamma)=\psi(g_0)=\psi(g)$ for every $g\in G$ and $\gamma\in \Gamma$.
We have established that  $\psi$ is a
quasi-character of $X$ that extends the character $\chi_0$.

Since $\psi(g(n))=\chi_0( g_0(n))$, we get
that $\norm{\psi}=\norm{\chi_0}\leq \delta^{-C}$ and also that
equation \eqref{E:quasi} is satisfied. Lastly, $\psi$ factors through
the affine torus $A$ and by Proposition~\ref{P:FK}, $A$ can be
identified with a finite dimensional torus $\T^l$. Under this
identification $\psi$ is mapped to a character of $\T^l$ and the
polynomial sequence $g(n)$ on the affine torus $A$ is mapped to a
polynomial sequence of unipotent affine transformations on $\T^l$.
This completes the proof.
\end{proof}

\section{A model  multiple recurrence result}\label{S:singlewarmup}
%% \subsection{A warmup multiple recurrence problem (Theorem
%%   B')}\label{SS:warmup}
We are going to prove Theorem~B'. For convenience, we repeat its statement:
\begin{theoremB'}
  Suppose that for every $r,m\in \N$ the set $S\subset \Z$
  contains patterns of the form $$\{r(c_{r,m}+mn)\colon 1\leq
  n\leq N_{r,m}\}$$ where $c_{r,m}, N_{r,m}$ are integers with $\lim_{m\to\infty}N_{r,m}=\infty$
  for every fixed $r\in \N$.

  Then $S$ is a set of multiple recurrence.
\end{theoremB'}

Part of the proof of Theorem~B' (the proof of Proposition
\ref{L:equidistribution} and the final step of the proof of Theorem~B'
in Section~\ref{SS:warmupconclusion}) carries almost verbatim to the
more complicated Hardy field setup (proof of Theorem A'). In order to
better illustrate the ideas we chose to give the argument in this
simpler setup. It splits in two parts, we first reduce things to
nilsystems and then verify a multiple recurrence property for
nilsystems.
%% \begin{proposition}
%%   Let $1<b<2$. Then there exist a set $S\subset \N$ with $d(S)>0$,
%%   and $c_m,N_m\in \N$ with $N_m\to\infty$ as $m\to\infty$, such
%%   that
%%$$
%%\{c_m+mn\colon 1\leq n\leq N_m, m\in S\} \subset \{[n^b]\colon n\in
%% \N\}.
%%$$
%% \end{proposition}
%% \begin{proof}
%%$$
%%(n_m+n)^b=n_m^b+bn n_m^{b-1}+\frac{b(b-1)}{2}n^2\xi_n^{b-2}
%%$$
%% \end{proof}
\subsection{Reduction to nilsystems}
We shall study the multiple ergodic averages that are naturally
associated to the multiple recurrence problem of Theorem~B'. We shall
show that the nilfactor is characteristic for $L^2$-convergence of
these averages.  Using Theorem~\ref{T:HoKra} it is
then not hard to see that in order to establish Theorem~$B'$ it
suffices to verify a multiple recurrence property for nilsystems.

As it is often the case when proving such reduction results, a key
tool is a Hilbert space
version of a classical elementary estimate of van der Corput. Its proof
is identical with the proof of this classical estimate (e.g. Lemma~3.1 in \cite{KN}).
\begin{lemma} \label{L:VDC} Let $v_1,\ldots, v_N$ be vectors
  of a Hilbert space with $\|v _i\| \leq 1$ for $i=1,\ldots,
  N$.

  Then for every integer  $H$ between $1$ and $N$ we have
$$
\norm{\frac{1}{N}\sum_{n=1}^N v_n}^2\leq 4 \cdot \Big(
\frac{1}{H}+\frac{H}{N}+ \frac{1}{H}\sum_{h=1}^H\Big|\frac{1}{N}
\sum_{n=1}^{N}\langle v_{n+h},v_n\rangle\Big|\Big).
$$
\end{lemma}

\begin{lemma}\label{L:seminorm} Let $(X,\mathcal{B},\mu,T)$ be a system and
  $f_1,f_2\in L^\infty(\mu)$
  satisfy $f_i\bot \mathcal{Z}$ for $i=1$ or $2$, where $\cZ$ is the nilfactor of the system.
Let $c_m,N_m$ be integers and $N_m\to \infty$ as $m\to\infty$.

   Then the averages
  \begin{equation}\label{E:S_M}
    \frac{1}{M}\sum_{m=1}^M\Big(\frac{1}{N_m}\sum_{n=1}^{N_m}   T^{c_m+mn}f_1
    \cdot T^{2(c_m+mn)}f_2\Big)
  \end{equation}
  converge to $0$ in $L^2(\mu)$ as $M\to \infty$.
\end{lemma}
\begin{proof}
  We can assume that $f_2\bot \mathcal{Z}$, the proof is similar in
  the other case. Furthermore, we can assume that $\norm{f_1}_{L^\infty}, \norm{f_2}_{L^\infty}\leq 1$.
    Let $A_M$ denote the averages \eqref{E:S_M}.  We
  have
$$
\norm{A_M}_{L^2(\mu)}^4 \leq \frac{1}{M}\sum_{m=1}^{M}\norm{
  \frac{1}{N_m}\sum_{n=1}^{N_m} T^{c_m+mn}f_1 \cdot\
  T^{2(c_m+mn)}f_2}_{L^2(\mu)}^4.
$$
Using Lemma~\ref{L:VDC} and the Cauchy-Schwarz inequality we get that
for every $H_{m,1}$ such that $H_{m,1}\prec N_m$ (meaning
$H_{m,1}/N_m\to 0$ as $m\to\infty$) the last expression is bounded by
\begin{align*}
  \frac{4}{M}\sum_{m=1}^{M}\frac{1}{H_{m,1}}\sum_{h_1=1}^{H_{m,1}}\Big|
  \frac{1}{N_m}\sum_{n=1}^{N_m}\int T^{c_m+mn}\bar{f_1} \cdot\
  T^{2(c_m+mn)}\bar{f_2}\cdot T^{c_m+mn+mh_1} f_1 \cdot &\
  T^{2(c_m+mn+mh_1)}f_2\ d\mu\Big|^2\\ &+o_{M,H_{m,1}}{(1)}.
\end{align*}
%% where $o_{M,H_{m,1}}(1)$ denotes an expression that goes to zero
%% when $H_{m,1}, M \to \infty$.
Factoring out the measure preserving transformation $T^{c_m+mn}$ and
using the Cauchy-Schwarz inequality we see that the last expression is bounded by
$$
  \frac{4}{M}\sum_{m=1}^{M}\frac{1}{H_{m,1}}\sum_{h_1=1}^{H_{m,1}}
  \norm{\bar{f_1}\cdot T^{mh_1} f_1}_{L^\infty(\mu)}\int  \Big|\frac{1}{N_m}
  \sum_{n=1}^{N_m}
  T^{c_m+mn}\bar{f_2} \cdot \
  T^{c_m+mn+2mh_1}f_2\Big|^2 \ d\mu+o_{M,H_{m,1}}{(1)}.
$$
Factoring out $T^{c_m}$ and using that $\norm{f_1}_{L^\infty}\leq 1$ we see that
the last expression is bounded by
$$
\frac{4}{M}\sum_{m=1}^{M}\frac{1}{H_{m,1}}\sum_{h_1=1}^{H_{m,1}}\norm{
  \frac{1}{N_m}\sum_{n=1}^{N_m} T^{mn}\bar{f_2}\cdot
  T^{mn+2mh_1}f_2}_{L^2(\mu)}^2+ o_{M,H_{m,1}}{(1)}.
$$
Using Lemma~\ref{L:VDC} again, factoring out $T^{mn}$, and
noticing that the resulting expression no longer depends on $n$, we
get that for every $H_{m,1},H_{m,2}\prec N_m$ this last expression is
bounded by
$$
\frac{4}{M}\sum_{m=1}^{M}\frac{1}{H_{m,1}}\sum_{h_1=1}^{H_{m,1}}
\frac{1}{H_{m,2}}\sum_{h_2=1}^{H_{m,2}}\Big|\int f_2 \cdot\
T^{2mh_1}\bar{f_1}\cdot T^{mh_2}\bar{f_2} \cdot\ T^{2mh_1+mh_2}f_2\
d\mu\Big| +o_{M,H_{m,1}, H_{m,2}}{(1)}.
$$
We can  choose $H_{m,1}, H_{m,2}$ to be $\prec N_m$, increase to $\infty$
as $m\to\infty$, and furthermore such that the subsets of
$\mathbb{N}^3$ defined by
$$
\Phi_{M}=\{ (m,h_1,h_2)\in \mathbb{N}^3\colon 1\leq m\leq M,\ 1\leq
h_1\leq H_{m,1},\ 1\leq h_2\leq H_{m,2}\}
$$
for $M\in\N$, form a F{\o}lner sequence.
%%(for example this would be the
%%case if $H_{m+1,1}- H_{m,1}$ and $H_{m+1,2}- H_{m,2}$ are bounded).
Since $f_2\bot \mathcal{Z}$, by
Corollary~\ref{C:L2'}  we have
$$
\frac{1}{|\Phi_M|}\sum_{(m,h_1,h_2)\in \Phi_M}\Big|\int f_2 \cdot\
T^{2mh_1}\bar{f_1}\cdot T^{mh_2}\bar{f_2} \cdot\ T^{2mh_1+mh_2}f_2\
d\mu\Big|
$$
converges to zero as $M\to\infty$. This shows that the averages  $A_M$ converge to
zero in $L^2(\mu)$ as $M\to \infty$ and finishes the proof.
\end{proof}
%% \begin{remark}
%%   It is actually possible to show that even if $f\bot
%%   \mathcal{Z}_2$ then the previous averages converges to zero. As a
%%   result when studying these averages we can restrict ourselves to
%%   $2$-step nilsystems. This case is simple enough to be handled
%%   ``by hand'', this way we can get the multiple recurrence result
%%   without great effort. However, this brute force approach we will
%%   not work for the general case so we will not pursue this venue.
%% \end{remark}
The proof of the next result is very similar to the proof of
Lemma~\ref{L:seminorm}, only notationally more complicated, and so we
omit it.
\begin{proposition}\label{P:seminorm'} Let $(X,\mathcal{B},\mu,T)$ be
  a system and  $f_1,\ldots, f_{\ell} \in L^\infty(\mu)$ satisfy
  $f_i\bot \mathcal{Z}$ for some $i=1,\ldots,\ell$, where $\cZ$ is the nilfactor of the system. Let $c_m,N_m$ be integers with $N_m\to  \infty$.

   Then the averages
$$
\frac{1}{M}\sum_{m=1}^M\Big(\frac{1}{N_m}\sum_{n=1}^{N_m}
T^{c_m+mn}f_1 \cdot T^{2(c_m+mn)}f_2\cdot\ldots\cdot T^{\ell
  (c_m+mn)}f_\ell \Big)
$$
converge to $0$ in $L^2(\mu)$ as $M\to \infty$.
\end{proposition}

\subsubsection{Dealing with nilsystems, a key equidistribution result}
\begin{lemma}\label{L:density0}
  Let $B\subset \R\setminus \Q$ and $K\subset \N$ be finite sets, and
  $(e_m)_{m\in\N}$ be a sequence of positive real numbers such that
  $\lim_{m\to\infty}e_m=0$.

   Then the set  $S=\{m\in \N\colon \norm{m^k\alpha}\leq
  e_m \text{ for some } \alpha\in B, \text{ and } k\in K\}$ has density $0$.
\end{lemma}
\begin{proof}
  If  $\alpha$ is irrational, then  the sequence $(m^k\alpha)_{m\in\N}$ is
  equidistributed in $\T$. Hence, for every $\varepsilon>0$ we have
  $d(\{m\in \N\colon\norm{m^k\alpha}\leq \varepsilon\})=2\varepsilon.$
  It follows that for fixed $k\in \N$ and $\alpha$ irrational, the set
  $S_{k,\alpha}=\{m\in \N\colon \norm{m^k\alpha}\leq e_m\}$ has zero
  density. Since $S$ is contained in a finite union of sets of the
  form $S_{k,\alpha}$ it also has zero density.
\end{proof}

Remember that given a connected nilmanifold $X=G/\Gamma$, an element
$a\in G$ is an ergodic nilrotation if the sequence
$(a^n\Gamma)_{n\in\N}$ equidistributed in $X$.
\begin{proposition}\label{L:equidistribution}
  Let $X=G/\Gamma$ be a connected nilmanifold  and $a\in G$ be an ergodic nilrotation. Let  $c_m$
  be positive integers and $(N_m)_{m\in\N}$ be a sequence of integers with
    $N_m\to \infty$.

  Then for every $F\in C(X)$
  we have
$$
\lim_{M\to\infty}\frac{1}{M}\sum_{m=1}^M\Big(\frac{1}{N_m}\sum_{n=1}^{N_m} F(a^{c_m+mn}\Gamma)\Big)= \int F \
dm_X.
$$
\end{proposition}
\begin{proof}
  It suffices to show that for every $\delta>0$, for a set of $m\in
  \N$ of density $1$, the finite sequence $(a^{c_m+mn}\Gamma)_{1\leq
    n\leq N_m}$ is $\delta$-equidistributed in $X$.

  Let $\delta>0$ be small enough, and suppose that the finite sequence
  $(a^{c_m+mn}\Gamma)_{1\leq n\leq N_m}$ is not
  $\delta$-equidistributed for some $m\in\N$.  By  Theorem~\ref{T:GT2}, there
   exist a constant $M=M_{\delta,X}$ ($M$ does not depend on $m, N_m$, or
  $c_m$) and a quasi-character $\psi$ with $\norm{\psi}\leq M$ such
  that
  \begin{equation}\label{E:M}
    \norm{\psi(a^{c_m+mn})}_{C^\infty[N_m]}\leq M.
  \end{equation}
  As explained in Section~\ref{S:GT2}, the affine torus $A$ of $X$ can
  be identified with a finite dimensional torus $\T^l$. After making
  this identification, we have $\psi(t)=\kappa\cdot t$ for some
  non-zero $\kappa\in \Z^l$, and the nilrotation $a$ induces a $d$-step
  unipotent affine transformation $T_a\colon \T^l\to \T^l$.
  %%Moreover,
  %%since $a$ is an ergodic nilrotation of $X$, the transformation
  %%$T_a$ is totally ergodic, and so the spectrum
%%$$
%%S=\{\beta\in \R\setminus \{0\} \colon T_af=e(\beta)\cdot f \text{ for
%%  some } f\in L^\infty(m)\}$$ of $T_a$ consists of irrational numbers.
Let $\pi(a)=(\alpha_1\Z,\ldots,\alpha_s\Z)$,  where $\alpha_i\in \R$,  be the projection of
$a$ on the horizontal torus $\T^s$.
Since $\pi(a)$ is an ergodic rotation the real numbers $1, \alpha_1,\ldots,\alpha_s$ are rationally independent.
%%Let $B=\{\beta_1,\ldots,\beta_s\}$ be a basis (over $\Q$) of $S$ (meaning,
%%non-trivial rational combinations of elements of $B$ are irrational).
The coordinates of $T_a^n e$, where $e$ is the identity element of $\T^l$,
are polynomials of $n$, and so $\kappa\cdot T_a^ne$ is a
polynomial of $n$. Moreover, it is not hard to see that the leading
term of the polynomial $\kappa\cdot T_a^ne$ has the from $ \alpha n^k$,
where $k\leq d$ and
\begin{equation}\label{E:beta}
  \alpha=\frac{1}{k!}\sum_{i=1}^s r_i\alpha_i, \ r_i\in \Z \text{ not all of them zero with } |r_i|\leq c_1\cdot M
\end{equation}
for some constant $c_1$ that depends only on $a$. From this and the
definition of $\norm{\cdot}_{C^\infty[N]}$ (see \eqref{E:norms}) it
follows that
$$
\norm{\psi(a^{c_m+mn})}_{C^\infty[N_m]}=\norm{\psi(T_a^{c_m+mn}e)}_{C^\infty[N_m]}\geq
N_m^k \norm{ m^k\alpha}.
$$
Combining this with \eqref{E:M} we get that
\begin{equation}\label{E:m}
  \norm{m^k\alpha}\leq \frac{M}{N^k_m}.
\end{equation}
Since $k\leq d$, and by \eqref{E:beta} we have only finitely many
options for (the irrational) $\alpha$, Lemma~\ref{L:density0} applies and
shows that the set of $m\in\N$ that satisfy equation \eqref{E:m} has
zero density. This shows that the finite sequence
$(a^{c_m+mn}\Gamma)_{1\leq n\leq N_m}$ is $\delta$-equidistributed in
$X$ for a set of $m\in\N$ with density $1$, completing the proof.
\end{proof}

\subsection{Conclusion of the argument} \label{SS:warmupconclusion} We
first  use Proposition~\ref{P:seminorm'} to carry out a reduction to
nilsystems step, and then use the equidistribution result of
Proposition~\ref{L:equidistribution} to verify a multiple recurrence
result for nilsystems. This will enable us  conclude the proof of
Theorem~B'. We first need one   lemma.
%%\begin{lemma}\label{L:rmulties}
 %%Suppose that $c_m,N_m$ are integers and $N_m\to \infty$ as $m\to\infty$.  Let
 %% $$ S=\{c_m+mn\colon 1\leq n\leq N_m, m\in \N\}.$$

%%Then for every $r,m\in\N$ there exist $c_{r,m},N_{r,m}\in \N$, with
%%$N_{r,m}\to\infty$ as $m\to\infty$, and such that
%%$$
%%S_r=\{r(c_{r,m}+mn),\ 1\leq n\leq N_{r,m},m\in \N\} \subset S.
%%$$
%%\end{lemma}
%%\begin{proof}
%%  Suppose that $(r,m)=d$, then $m=dm_1$ for some $m_1\in \N$ such that
%%  $(m_1,r)=1$. Choose $1\leq k\leq r$ such that $km_1\equiv -c_m
%%  \pmod{r}$.  Then $c_m+m_1 (drn+k)=r(c_{r,m}+mn)$ for some
%%  $c_{r,m}\in \N$, and so $ r(c_{r,m}+mn)\in S \text{ for } 1\leq n\leq
%%  K_{m}$, where $K_m=(N_{m_1}-r)/(dr)$.  The result follows.
%%\end{proof}

\begin{lemma}\label{L:Ziegler}
Let $X=G/\Gamma$ be a connected nilmanifold and $a\in G$ be an ergodic nilrotation.

%%Then for every $\ell\in \N$ there exists a connected sub-nilmanifold $Y$
%%of $X^\ell$  such that for a.e. $g\in G$
%%the sequence $(a^n_g\Gamma,a^{2n}_g\Gamma,\ldots,a^{\ell n}_g\Gamma)_{n\in\N}$ is equidistributed on the %%nilmanifold
%%$Y$, where $a_g=g^{-1}ag$.

Then there exists a connected sub-nilmanifold $Z$ of $X^\ell$ such that for a.e. $g\in G$ the element $b_g=(g^{-1}ag,g^{-1}a^2g,\ldots,g^{-1}a^\ell g)$ acts ergodically on $Z$.
\end{lemma}
\begin{remark}
The independence of $Z$ on the generic $g\in G$ will not be needed, only that $Z$ is connected will be used.
\end{remark}
 \begin{proof}
This is an easy consequence of a limit formula that appears  in Theorem~2.2 of \cite{Z0}
(the details of the deduction appear in Corollary~2.10 of \cite{Fr}).
\end{proof}
We are now ready to give the proof of Theorem~B'.
\begin{proof}[Proof of Theorem~B']
  Fix $\ell\in\N$. For $r\in\N$ let
$$
\{r(c_{r,m}+mn)\colon 1\leq
  n\leq N_{r,m}\}
  $$
be the subset of $S$ defined
  in the statement of Theorem~B'. It suffices to show that for every system
  $(Y,\mathcal{B},\mu,T)$, and $f\in L^\infty(\mu)$ non-negative and
  not a.e. zero, there exists an $r\in\N$ such that
  \begin{equation}\label{E:positive}
    \liminf_{M\to\infty} \frac{1}{M}\sum_{m=1}^M\Big(\frac{1}{N_{r,m}}\sum_{n=1}^{N_{r,m}}   \int f \cdot T^{r(c_{r,m}+mn)}f
    \cdot\ldots\cdot T^{\ell r(c_{r,m}+mn)}f d\mu \Big)>0.
  \end{equation}

  Suppose first that the system is ergodic.
  %%We start with some reductions.  Using an ergodic decomposition
  %%argument we see that in order to establish \eqref{E:positive} for
  %%some fixed $r\in \N$ (to be chosen later)  we can assume that the system is ergodic.
  Using a slight modification of
  Proposition~\ref{P:seminorm'}, we see that the nilfactor
  $\mathcal{Z}$ is characteristic for the multiple ergodic averages
  appearing in \eqref{E:positive} (remember $\cZ$ contains all factors
  $\cZ_\ell$ for $\ell\in\N$). Therefore,  it suffices to verify
  \eqref{E:positive} with $\E(f|\mathcal{Z})$ in place of $f$.  As a consequence,
   by
  Theorem~\ref{T:HoKra},  we can assume that our system
  is an inverse limit of  nilsystems.
  %% , we have reduced the problem to establishing \eqref{E:positive}
  %% for inverse limits of $k$-step nilsystems.  Moreover, an argument
  %% completely analogous to that of Lemma $3.2$ in \cite{FuK} shows
  %% the existence of an $r\in\N$ such that \eqref{E:positive} holds
  %% is preserved by inverse limits.

  In this case, for given $\varepsilon>0$ (to be specified later)
  there exists a finite step nilfactor $\mathcal{N}$ such that
  $h=\E(f|\mathcal{N})$ satisfies $\norm{f-h}_{L^2(\mu)}\leq
  \varepsilon$.  It is easy to verify that
  \begin{equation}\label{E:appr}
    \Big|\int f\cdot T^{n}f\cdot\ldots\cdot T^{\ell n}f\ d\mu-\int h\cdot T^{n}h\cdot\ldots\cdot T^{\ell n}h\ d\mu\Big|\leq c_1\varepsilon
  \end{equation}
  for every $n\in\N$, where $c_1$ is some absolute constant that
  depends only on $f$.  Using an appropriate conjugation we can assume that $T=T_a$ is an ergodic  nilrotation acting on a nilmanifold $X$, $\mu=m_X$, and $h$ is a non-negative,  bounded measurable function on $X$,
  with $\int h \ dm_X=\int f \ d\mu$. We are going to work with these extra assumptions henceforth.

 Let $X_0$ be the connected component of the nilmanifold $X$. It is easy to see that   there exists an $r_0\in \N$ such that the nilmanifold $X$ is the disjoint union of the connected sub-nilmanifolds $X_i=a^iX_0$, $i=0,\ldots,r_0-1$, and $a^{r_0}$ acts ergodically on each $X_i$.

 For $r=r_0$,  we shall see that  \eqref{E:positive}  follows easily from Szemer\'edi's theorem and the identity
\begin{multline}\label{E:r_0}
\lim_{M\to\infty}\frac{1}{M}\sum_{m=1}^M\Big(\frac{1}{N_{r,m}}\sum_{n=1}^{N_{r,m}}   \int h(x) \cdot h(a^{r_0(c_{r,m}+mn)}x)
    \cdot\ldots\cdot h(a^{\ell r_0(c_{r,m}+mn)}x) \  dm_X \Big)=\\
\lim_{N\to\infty}  \frac{1}{N}\sum_{n=1}^N  \int h(x) \cdot h(a^{r_0 n}x)
    \cdot\ldots\cdot h(a^{\ell r_0 n}x) \ dm_X.
  \end{multline}
We first verify \eqref{E:r_0}.   An easy approximation argument shows that  it suffices to  verify \eqref{E:r_0} for every $h\in C(X)$.
  Our plan is to use Lemma~\ref{L:equidistribution} to establish a stronger pointwise result.
We can assume that  $x=g\Gamma$ is an element of $X_0$, a similar argument applies if $x\in a^i X_0$ for
$i=1,\ldots,r_0-1$. An easy computation shows that the limit
\begin{equation} \label{E:r_0A}
\lim_{M\to\infty}\frac{1}{M}\sum_{m=1}^M\Big(\frac{1}{N_{r,m}}\sum_{n=1}^{N_{r,m}}   \ h(a^{r_0(c_{r,m}+mn)}x)
    \cdot\ldots\cdot h(a^{\ell r_0(c_{r,m}+mn)}x)\Big)
    \end{equation}
    is equal to the limit
    \begin{equation}\label{E:asar}
   \lim_{M\to\infty} \frac{1}{M}\sum_{m=1}^M\Big(\frac{1}{N_{r,m}}\sum_{n=1}^{N_{r,m}}   \tilde{h}(b^{(c_{r,m}+mn)}_g\tilde{\Gamma})\Big)
    \end{equation}
  where $$
  b_g=(g^{-1}a^{r_0}g,g^{-1}a^{2 r_0}g,\ldots, g^{-1}a^{\ell r_0}g),
   $$
   $\tilde{h}(x_1,\ldots,x_l)=h(gx_1)\cdot \ldots\cdot h(gx_\ell)$ ($\in C(X^\ell)$),
  and $\tilde{\Gamma}=\Gamma^\ell$.
  Since $a^{r_0}$ acts ergodically on the connected nilmanifold $X_0$, by Lemma~\ref{L:Ziegler} there exists a connected sub-nilmanifold $Z$ of $X^\ell_0$ such
  that for a.e. $g\in G$ the element
  $b_g$ acts ergodically on $Z$. For those values of $g$,   Lemma~\ref{L:equidistribution} gives that  the limit  \eqref{E:asar} is equal to $\int \tilde{h} \  dm_Z$. Since $b_g$ acts ergodically on $Z$, this integral is also   equal to the limit
  $$
  \lim_{N\to\infty}\frac{1}{N}\sum_{n=1}^N \tilde{h}(b^n_g\tilde{\Gamma})
$$
which can be rewritten as
  \begin{equation}\label{E:r_0B}
  \lim_{N\to\infty}  \frac{1}{N}\sum_{n=1}^N   h(a^{r_0 n}x)
    \cdot\ldots\cdot h(a^{\ell r_0 n}x).
  \end{equation}
   We have established the equality of the limits
\eqref{E:r_0A} and \eqref{E:r_0B} for a.e. $x\in X_0$. As we
mentioned, a similar argument applies for a.e. $x\in X$ and this
readily implies \eqref{E:r_0}.

Next we use \eqref{E:r_0} to establish \eqref{E:positive}. In this regard, we estimate the limit
\begin{equation}\label{E:simplified}
\lim_{N\to\infty}  \frac{1}{N}\sum_{n=1}^N  \int h(x) \cdot h(a^{r_0 n}x)
    \cdot\ldots\cdot h(a^{\ell r_0 n}x) \ dm_X.
  \end{equation}
Since the function $h$ is a.e. non-negative, using a uniform version of  Furstenberg's
multiple recurrence theorem (see e.g. \cite{BHRF}) we get that the limit \eqref{E:simplified}
is bounded from below by a positive constant $c_2$ that depends only
on $\int h\ dm_X=\int f d\mu$ (and is independent of $r_0$).  Combining this with \eqref{E:appr}
and \eqref{E:r_0}, we get that
$$
\liminf_{M\to\infty} \frac{1}{M}\sum_{m=1}^M\Big(\frac{1}{N_{r_0,m}}\sum_{n=1}^{N_{r_0,m}}
\int f \cdot T^{r_0(c_{r_0,m}+mn)}f \cdot\ldots\cdot T^{\ell
  r_0(c_{r_0,m}+mn)}f d\mu\Big) \geq c_2-c_1\varepsilon.
$$
Since the positive constants $c_1,c_2$ depend only on $f$ we can
choose $\varepsilon<c_2/c_1$ and verify \eqref{E:positive} for
$r=r_0$.

To deal with the general case, we use an ergodic decomposition argument. For a.e. ergodic component we can use the previous argument to find an $r\in \N$ for which \eqref{E:positive} holds. Since there are only countably many choices for $r$, there exists an $r_0\in \N$ for which  \eqref{E:positive} holds for a set of ergodic components that has positive measure.   The result follows.
\end{proof}

%% \begin{remark}
%%   It is possible to prove this result without restricting our
%%   averaging scheme on the very beginning to some subset $S_r$ of
%%   $S$. We choose not to do so, and %%present this
%%   slightly more complicated argument, because it is closer in
%%   spirit to the one used later for Hardy %%sequences.
%% \end{remark}

\section{Polynomial structure for Hardy
  sequences}\label{S:polypattern}
\subsection{Result and idea of the proof} In this section we shall show
that if the function $a\in\H$ satisfies $x^k\prec a(x) \prec x^{k+1}$
for some non-negative integer $k$, then the range of the sequence
$[a(n)]$ contains some suitably chosen polynomial patterns. We are going to
work with these patterns in the next section in order to prove
Theorem~A'.
\begin{proposition}\label{P:polypattern} Let $a\in \H$ be eventually
  positive and satisfy $x^{k}\prec a(x)\prec x^{k+1}$ for some non-negative
  integer $k$.

  Then for every $r\in \N$, and every large enough $m\in\N$,
  there exist polynomials
  $p_{r,m}(n)$ of degree at most $k-1$, and $N_{r,m}\in\N$ with
  $N_{r,m}\to\infty$ $($as $m\to\infty$ and $r$ is fixed$)$, such that
$$
\{r(mn^k+p_{r,m}(n)), 1\leq n\leq N_{r,m}\} \subset \{[a(n)]\colon
n\in \N\}.
$$
\end{proposition}
\begin{remark}For $k=1$ we get the same patterns as in Theorems~B
  and B'.
  %%For $k>1$ the presence of the factors $r$ is needed
  %%since a result analogous to Lemma~\ref{L:rmulties} does not hold.
\end{remark}
The initial idea is rather simple. Let us illustrate it for $r=k=1$. In this case
$x\prec a(x)\prec x^2$ and
 we are searching to find arithmetic progressions of the form
 $$
 \{c_m+mn, 1\leq n\leq N_m\}
$$
within the range of $[a(n)]$ for some $c_m,N_m\in\N$ with
$N_m\to\infty$.  Using the Taylor expansion of the function $a(x)$ around an integer
$n_m$ (to be specified later) we get
\begin{equation}\label{E:simpletaylor}
  a(n_m+n)=a(n_m)+a'(n_m)n+a''(\xi_n)n^2/2
\end{equation}
for some $\xi_n\in [n_m,n_m+n]$. Since $a\in \H$ is eventually positive and $x\prec a(x)\prec
x^2$, it is easy to see that $1\prec a'(x)\prec x$, $a''(x)>0$, and
$a''(x)\to 0$ (see Lemma~\ref{L:basic}).  Since $a'\in\H$, the
estimate for $a'$ easily implies that the range of the sequence
$[a'(n)]$ is a cofinite subset of $\N$ (see Lemma~\ref{L:basic}).  Therefore,
for large $m$ there exists $n_m\in \N$ such that $[a'(n_m)]=m$. Since
$n_m\to \infty$, we have  $a''(\xi_n)\to 0$
(also $a''(\xi_n)>0$). Therefore, if we could choose $n_m$ so that in addition
to $[a'(n_m)]=m$ we have that the fractional parts of the numbers
$a(n_m)$ and $a'(n_m)$ converge to $0$ as $m\to\infty$, then
\eqref{E:simpletaylor} would give that
$$
[a(n_m+n)]=[a(n_m)]+[a'(n_m)]n=c_m+mn, \quad c_m=[a(n_m)]
$$
for every $n\in [1,N_m]$, for some $N_m\in\N$ with $N_m\to\infty$. This
is exactly what we wanted.

To carry out this plan we shall need an equidistribution result that
will enable us to get the ``small fractional parts'' assumption that
we mentioned before.  We are going to prove this by using a classical estimate
of van der Corput on oscillatory exponential sums.

\subsection{Some basic properties}
As we explained before,  if $a\in \H$ and $b\in \LE$, then the limit $\lim_{x\to\infty} a'(x)/b'(x)$ exists
(possibly infinite). Hence, if both  functions  $a(x), b(x)$ converge to $0$, or to $\infty$,
then by L'Hospital's rule we have   $\lim_{x\to\infty} a(x)/b(x)=\lim_{x\to\infty} a'(x)/b'(x)$.
  We are going to make use of this fact to  prove
some  basic properties of elements of $\H$ that we
shall frequently use:
\begin{lemma}\label{L:basic}
  Let $a\in \H$ be eventually positive and satisfy $x^{k} \prec
  a(x)\prec x^{k+1}$ for some non-negative integer $k$.
  Then  \\
  $(i)$ $x^{k-l}\prec a^{(l)}(x)\prec x^{k+1-l}$ for every $l\leq k$,
  %% for $l\leq k+1$, and $a^{(l)}(x)\prec x^{k+1-l+\delta}$ for every
  %% $\delta>0$ for $l>k+1$.
  $x^{-1-\delta}\prec a^{(k+1)}(x)\prec 1$ for every
  $\delta>0$,\footnote{Take $a(x)=\log x$ or $a(x)=\log\log x$ and
    $k=0,l=1$, to see the necessity of introducing the term $\delta$.}
  and $a^{(l)}(x)$ is eventually positive for every $l\leq k+1$.\\
  $(ii)$ $a^{(l)}(x)$
  is eventually increasing for $0\leq l\leq k$, and decreasing for $l=k+1$.\\
  $(iii)$ If $k\neq 0$ for every non-zero $c\in\R$ we have
  $x^{k-1}\prec a(x+c)-a(x)\prec x^k$.
  If $k=0$, then $a(x+c)-a(x)\prec 1$ and the range of $[a(n)]$ is a cofinite subset of $\N$.\\
  $(iv)$ For every $c\in \R$ we have $a(x+c)/a(x)\to 1$ as
  $x\to\infty$. More generally the same holds if $c=c(x)$ is a bounded
  function.
\end{lemma}
\begin{proof}
  $(i)$ All parts are a direct consequence of L'Hospital's rule.

  %% First notice that in order to deal with the first part it
  %% suffices to show that $x^{k-1}\prec a'(x)\prec %% x^k$.
  %% Indeed, if $k-1\geq 1$ applying the $l=1$ case to $a'\in\H$ we
  %% get $x^{k-2}\prec a''(x)\prec %%x^{k-1}$. Continuing like that we get the advertised estimate.

  %% We show that $a'(x)\prec x^k$. Suppose that this is not the case.
  %% Since $\lim_{x\to\infty} a'(x)/ x^k$ exists, it has to be a
  %% positive real number or $+\infty$ (it cannot be negative or
  %% $-\infty$ because then $a$ would be eventually decreasing
  %% contradicting that $ %%a(x)\to +\infty$). In either case we get for large $x$ that $a'(x)> cx^k$ for some positive constant $c$. %%Integrating we get
  %% a contradiction to the fact that $a(x)\prec x^{k+1}$.  Similarly
  %% we show that $x^{k-1}\prec a'(x)$.

  %% By the first we have $1\prec a^{(k)}(x)\prec x$ . Using this and
  %% arguing as before we easily get the %%advertised estimate.

  $(ii)$ If $l\leq k$ the result follows immediately from part $(i)$.
  If $l=k+1$ by part $(i)$ we have that $a^{(k+1)}(x)\to 0$ and
  $a^{(k+1)}(x)$ is eventually positive.  It follows that
  $a^{(k+1)}(x)$ is eventually decreasing.

  $(iii)$ Let $k\neq 0$. By part $(i)$ we have that $x^{k-1}\prec
  a'(x)\prec x^k$. The result now follows from the mean value theorem.
  Suppose now that $k=0$. By part $(i)$ we have $a'(x)\prec 1$ and the
  mean value theorem gives that $a(x+c)-a(x)\prec 1$. Since
  $a(n)\to+\infty$ and $a(n+1)-a(n)\to 0 $ it follows that the range
  of the sequence $[a(n)]$ is a cofinite subset of $\N$.

  $(iv)$ Notice that $a(x+c)=a(x)+b(x)$ where $b(x)=a(x+c)-a(x)\prec
  a(x)$ by part $(iii)$. It follows that $a(x+c)/a(x)\to 1$.
  %% Since $a(x+c)\in \H$ and $a(x+c)\to \infty$ we have
  %% $g(x)=\log(a(x+c))\in \H$. Moreover, since $a(x)\prec x^k$ we
  %% have $g(x)\prec x$, so $g'(x)\prec 1$ by $(i)$. Using the mean
  %% value theorem we get %%$g(x+c)-g(x)\to 0$, or equivalently that
  %% $a(x+c)/a(x)\to 1$ as $x\to\infty$.
  Suppose now that $c=c(x)$ satisfies $|c(x)|\leq M$.  Since $a(x)$ is
  eventually positive and increasing, we have
   $ a(x-M)/a(x)\leq
  a(x+c(x))/a(x)\leq a(x+M)/a(x)$ for all large enough $x$. The result now follows from the
  case where $c(x)$ is constant.
\end{proof}
Given some growth estimates for $a\in \H$ we shall derive some
estimates about the compositional inverse $a^{-1}$ of $a$ (which is
not necessarily in $\H$).
\begin{lemma}\label{L:basic'}
  Let $a\in \H$ be eventually positive and satisfy $x^{\delta} \prec
  a(x)\prec x$ for some $\delta\in (0,1)$.

  Then  \\
  $(i)$ $x\prec a^{-1}(x)\prec x^{1/\delta}$.\\
  $(ii)$  $(a^{-1})'$  is eventually increasing,  $1\prec (a^{-1})'(x)\prec x^{1/\delta-1}$, and $(a^{-1})''(x)\prec 1$.\\
  $(iii)$ For every $c\in\R$ we have $(a^{-1})'(x+c)/(a^{-1})'(x)\to
  1$ as $x\to\infty$.
\end{lemma}
\begin{proof}
  $(i)$ By Lemma~\ref{L:basic} we have that $a$ is eventually
  increasing, so the same is true for $a^{-1}$.  Hence,  our hypothesis
  gives $a^{-1}(x^{\delta}) \prec x\prec a^{-1}(x)$, which implies the
  advertised estimate.

  $(ii)$ By Lemma~\ref{L:basic} we have that $a'$ is eventually
  decreasing.  Since
  \begin{equation}\label{E:inv1}
    (a^{-1})'(x)=\frac{1}{a'(a^{-1}(x))}
    %% , \quad (a^{-1})''(x)=\frac{-a'''(a^{-1}(x)) \cdot
    %%   (a^{-1})'(x)}{(a''(a^{-1}(x)))^2}.
  \end{equation}
  and $a^{-1}$ is eventually increasing, it follows immediately that
  $(a^{-1})'$ is eventually increasing.

  Using L'Hospital's rule we get
  \begin{equation}\label{E:inv2}
    x^{-1+\delta} \prec a'(x)\prec 1.
  \end{equation}
  Combining \eqref{E:inv1}, \eqref{E:inv2}, and part $(i)$, we get the
  first estimate, and similarly we deal with the second.

  $(iii)$ Using the estimates in $(ii)$, the proof is the same as in
  part $(iii)$ and $(iv)$ of Lemma~\ref{L:basic}.
\end{proof}

\subsection{An equidistribution result}
In this subsection we shall establish the equidistribution result
needed for the proof of Proposition~\ref{P:polypattern}.  We first
state and prove it in its simplest form (Lemma~\ref{L:equi1}), and subsequently we prove a
more technical variation (Proposition~\ref{L:equi2}) that is better
suited for our purposes.

 The following estimate is crucial for the
results in this subsection. The proof can be found in \cite{KN}
(Theorem 2.7).

\begin{lemma}[van der Corput~\cite{VdC}]\label{L:VDC2}
  Let $k,l$ be integers with $k<l$ and let $f$ be twice differentiable
  on $[k,l]$ with $f''(x)\geq \rho>0$ or $f''(x)\leq -\rho<0$ for
  $x\in [k,l]$.

  Then
$$
\sum_{n=k}^l e(f(n))\leq
(|f'(l)-f'(k)|+2)\Big(\frac{4}{\sqrt{\rho}}+3\Big).
$$
\end{lemma}

\begin{definition}Let $(I_m)_{m\in\N}$ be a sequence of intervals of
  integers (with lengths going to $\infty$). We say that the sequence
  $a(n)$ with values in $\R^d$, is \emph{equidistributed in $\T^d$
    with respect to the intervals $I_m$}, if for every Riemann
  integrable function $\phi\colon \T^d\to \C$ we have
$$
\lim_{m\to\infty} \frac{1}{|I_m|} \sum_{n\in I_m}
\phi(a(n))=\int_{\T^d} \phi \ dm_{\T^d}.
$$
\end{definition}
As it is well known, it suffices to verify the previous identity for
every non-trivial character $\phi=\chi$ of $\T^d$ (in which case the
integral is $0$).
\begin{lemma}\label{L:equi1}
  Suppose that $a\in\H$ is eventually positive and satisfies $x\prec
  a(x)\prec x^{2}$.

  Then for every $\varepsilon>0$ there exists a
  sequence of intervals
  $(I_m)_{m\in\N}$  such that \\
  $(i)$ $m\leq a'(n)\leq m+\varepsilon$, for every $n\in I_m$ and $m$ large enough, and \\
  $(ii)$ the sequence $a(n)$ is equidistributed in $\T$ with respect
  to the intervals $I_m$.
\end{lemma}
\begin{proof}
  Let $\varepsilon>0$. Suppose for the moment that $k_m,l_m\in\N$ have been chosen so
  that the intervals $I_m=[k_m,l_m]$, satisfy condition $(i)$. Let us
  see how we deal with condition $(ii)$. We need to guarantee that for
  every non-zero integer $s$ we have
  \begin{equation}\label{E:b_m}
    \lim_{m\to\infty} \frac{1}{l_m-k_m}\sum_{n=k_m}^{l_m} e(sa(n))=0.
  \end{equation}
  To estimate the  average in \eqref{E:b_m} we are going to use Lemma~\ref{L:VDC2}.
  Since $m\leq a'(n)\leq m+\varepsilon$ for $n\in I_m$, we have
  $|a'(l_m)-a'(k_m)|\leq \varepsilon$, and since $a''(x)$ is
  eventually decreasing (by Lemma~\ref{L:basic}) we have for large $m$
  that $\rho=a''(l_m)$ satisfies the assumptions of
  Lemma~\ref{L:VDC2}.  We get
  %% \begin{equation}\label{E:b_m}
$$
\frac{1}{l_m-k_m}\sum_{n=k_m}^{l_m} e(sa(n))\leq \frac{2+\varepsilon
  s}{l_m-k_m} \cdot \Big(\frac{4}{\sqrt{sa''(l_m)}}+3\Big).
$$
%% \end{equation}
Therefore,  in order to establish \eqref{E:b_m} it suffices to show that
\begin{equation}\label{E:l1}
  \lim_{m\to\infty}\frac{1}{(l_m-k_m)\sqrt{a''(l_m)}}=0.
\end{equation}
We shall now make a choice of $k_m,l_m\in\N$ so that conditions $(i)$
and $\eqref{E:l1}$ are satisfied.  Notice that Lemma~\ref{L:basic} implies that $a'(x)$ increases to  $+\infty$. We consider two cases:

{\bf Case 1.} Suppose that $a(x)\prec x^{1+\delta}$ for every
$\delta>0$.  Let $k_m$ be the first integer such that $a'(n)\geq m$,
and define
$$
I_m=[k_m,k_m+k_m^{\frac{3}{4}}].
$$
%% Since $k_m\to\infty$ we have that $|I_m|\to\infty$, so condition
%% $(i)$ is satisfied.
We first show that condition $(i)$ is satisfied. Choose any $\delta\in
(0,1/4)$. Since $a(n)\prec n^{1+\delta}$, arguing as in
Lemma~\ref{L:basic} we get $a''(n)\prec n^{-1+\delta}$. Using the mean
value theorem and the fact that $a'(x)$ is eventually increasing (by
Lemma~\ref{L:basic}) we get
\begin{equation}\label{E:D1}
  \max_{n\leq k_m^{3/4}}(a'(k_m+n)-a'(k_m))= a'(k_m+ k_m^{\frac{3}{4}})-a'(k_m)= k_m^{\frac{3}{4}} \cdot a''(\xi_m)\prec
  k_m^{\frac{3}{4}}\cdot k_m^{-1+\delta}\to 0.
\end{equation}
Furthermore, since $a'(n+1)-a'(n)\to 0$ (by Lemma~\ref{L:basic}), form the definition of $k_m$
 we have that
$a'(k_m)\to m$ as $m\to\infty$.  From this and \eqref{E:D1} it follows
that condition $(i)$ is
satisfied.

It remains to verify \eqref{E:l1}. Since $a(n)\succ
n$ we get by Lemma~\ref{L:basic} that $a''(n) \succ n^{-1-\delta}$ for
every $\delta>0$. Using this, and keeping in mind that $\delta<1/4 $
we find that
$$
\frac{1}{(l_m-k_m)\sqrt{a''(l_m)}}\ \prec
\frac{({k_m+k_m^{\frac{3}{4}})^{\frac{1+\delta}{2}}}}{k_m^{\frac{3}{4}}}\prec
k_m^{-1/8}\to 0.
$$
This proves \eqref{E:l1} and completes the proof of Case $1$.

{\bf Case 2.} Suppose that $x^{1+\delta}\prec a(x)\prec x^2$ for some
$\delta>0$.  Using Lemma~\ref{L:basic} we find that $x^\delta \prec a'(x)\prec x$
so we can apply Lemma~\ref{L:basic'} for $a'$ in place of $a$.
For convenience we set
$$b(x)=(a')^{-1}(x)$$ and
summarize some properties that follow from Lemmas~\ref{L:basic},
\ref{L:basic'} and  will be used later
\begin{equation}\label{E:D2}
  a'(n+1)-a'(n)\to 0, \quad b'(n) \text{ increases to }  \infty ,\quad  b'(n+c)/b'(n)\to 1 \text{ for every } c\in\R.
\end{equation}

We define
$$
I_m=\{n\in\N\colon m\leq a'(n)\leq m+\varepsilon \}=[k_{m},l_{m}].
$$
Obviously, condition $(i)$ is satisfied, therefore
%% Since $b'(n)\prec 1$ we have $b(n+1)-b(n)\to 0$ (see
%% Lemma~\ref{L:basic}) as $n\to\infty$. Hence, %%$|I_m|\to\infty$ as $m\to\infty$, which shows
%% that condition $(i)$ is also satisfied.
it remains to verify \eqref{E:l1}.
Since $a'(n+1)-a'(n)\to 0$ (by \eqref{E:D2}) we get from the definition of $k_m$ that
 $a'(k_m)-m\to 0$.
It follows that $b(m)- k_m\to 0$, and similarly we get
$b(m+\varepsilon)- l_m\to 0$. Since $b'$ is eventually increasing (by
\eqref{E:D2}), using the mean value theorem we get for large $m$ that
\begin{equation}\label{E:l2}
  l_m-k_m\geq \frac{1}{2}\cdot (b(m+\varepsilon)-b(m))= \frac{\varepsilon}{2}\cdot b'(\xi_m)\geq \frac{\varepsilon}{2}\cdot b'(m).
\end{equation}
Furthermore, since $b'(x) =\frac{1}{a''(b(x))}$, or equivalently $a''(b(x))=\frac{1}{b'(x)}$, setting
$x=m+\varepsilon$ and using that $b(m+\varepsilon)-l_m\to 0$ gives
$$
a''(l_m)- \frac{1}{b'(m+\varepsilon)}\to 0.
$$
It follows that for large $m$ we have
\begin{equation}\label{E:l3}
  a''(l_m)\geq \frac{1}{2 \, b'(m+\varepsilon)}.
\end{equation}
Combining \eqref{E:l2} and \eqref{E:l3} we get for large $m$ that
$$
\frac{1}{(l_m-k_m)\sqrt{a''(l_m)}}\leq C_2\cdot
\frac{\sqrt{b'(m+\varepsilon)}}{b'(m)}=C_2\cdot
\sqrt{\frac{{b'(m+\varepsilon)}}{b'(m)}}\cdot \frac{1}{\sqrt{b'(m)}}
$$
where $C_2=2^{3/2}/\varepsilon$. The last expression converges to zero
as $m\to\infty$ since the first fraction converges to $1$ and
$b'(m)\to \infty$ (by \eqref{E:D2}). Hence, we have established
\eqref{E:l1}, completing the proof of Case $2$.

Since the two cases cover all the functions $a\in \H$ that satisfy
$x\prec a(x)\prec x^2$ the proof is complete.
\end{proof}

We now derive an extension of Lemma~\ref{L:equi1} that will be used
later.  A big part of the proof is analogous to that of
Lemma~\ref{L:equi1} so we are  just going to sketch it.
\begin{proposition}\label{L:equi2}
  Suppose that $a\in \H$ is eventually positive and satisfies
  $x^k\prec a(x)\prec x^{k+1}$ for some $k\in\N$.
Let   $\varepsilon>0$  and $d_0,d_1,\ldots,d_k\in\N$.

  Then   there exists a sequence of intervals
  $(I_m)_{m\in\N}$  such that \\
  $(i)$ $d_k m\leq a^{(k)}(n)\leq d_k m+\varepsilon$, for every $n\in I_m$ and $m$ large enough, and \\
  $(ii)$ the sequence $(a(n)/d_0,a'(n)/d_1,\ldots,
  a^{(k-1)}(n)/d_{k-1})$ is equidistributed in $\T^{k-1}$ with respect
  to the sequence of intervals $I_m$.
\end{proposition}
\begin{proof}
  We assume that $d_0=d_1=\ldots=d_k=1$, the general case is similar.
  As in the proof of Lemma~\ref{L:equi1} we define the sequence of intervals $I_m=[k_m,l_m]$ as follows: \\
  {\bf Case 1}: If $a(x)\prec x^{k+\delta}$ for every $\delta>0$ (in
  which case $a^{(k)}(x)\prec x^\delta$ for every $\delta>0$), then
$$
I_m=[k_m,k_m+k_m^{3/4}] \text{ where } k_m \text{ is the smallest }
n\in \N \text{ such that } a^{(k)}(n)\geq m.
$$
{\bf Case 2}: If $ a(x)\succ x^{k+\delta}$ for some $\delta>0$ (in
which case $x^{\delta}\prec a^{(k)}(x)\prec x$ for some $\delta>0$),
then
$$
I_m=\{n\colon m\leq a^{(k)}(n) \leq m+\varepsilon\}.
$$

Arguing as in Lemma~\ref{L:equi1} we can show that condition $(i)$ is
satisfied.  Therefore, it remains to show the equidistribution property
$(ii)$, or equivalently, that for $c_0,\ldots,c_k\in \Z$, not all of
them zero, the sequence $b(n)$ defined by
$$
b(n)=c_0a(n)+c_1a'(n)+\ldots +c_{k-1}a^{(k-1)}(n)
$$ is equidistributed in $\T$
with respect to the sequence of intervals $I_m$.  To do this we shall
use a difference theorem of van der Corput (Theorem $3.1$ in
\cite{KN}) which enables us to reduce matters  to a setup similar to the
one treated in Lemma~\ref{L:equi1} (a similar trick was used in
\cite{Bo3}). We are going to  assume that $c_0\neq 0$,  the other cases can
be treated similarly.

By the theorem of van der Corput, in order to show that $b(n)$ is
equidistributed in $\T$ with respect to the sequence of intervals
$I_m$, it suffices to show that for every $m\in\N$ the sequence
$\Delta_mb(n)$, where $\Delta_mb(n)=b(n+m)-b(n)$, is equidistributed
in $\T$ with respect to the sequence of intervals $I_m$. Applying this
successively we reduce our problem to showing that for every
$m_1,\ldots,m_{k-1}\in\N$ the sequence $B(n)$, where the function $B\in \H$ is defined by
$$
B(x)=\Delta_{m_1}\Delta_{m_2}\cdots\Delta_{m_{k-1}}b(x),
$$
is equidistributed in $\T$ with respect to the sequence of intervals
$I_m$.
%% Equivalently, we need to establish that for every nonzero integer
%% $s$ we have
%% \begin{equation}\label{E:b_m}
%%\lim_{m\to\infty} \frac{1}{l_m-k_m}\sum_{n=k_m}^{l_m} e(sA(n))=0.
%%\end{equation}

In order to prove this, we first derive some  properties about the function $B(x)$ that will be useful.  Since $x^k\prec b(x) \prec x^{k+1}$, by
repeatedly applying Lemma~\ref{L:basic} we get
\begin{equation}\label{E:A2}
  x\prec B(x)\prec x^2.
\end{equation}
It will also be useful to relate the functions $B'(x)$ and
$a^{(k)}(x)$.  By repeatedly applying the mean value theorem and using
that $a^{(k+1)}(x)\to 0$ (follows from Lemma~\ref{L:basic}) we get
\begin{equation}\label{E:B'}
  \lim_{x\to\infty} (B'(x)-M a^{(k)}(x+\xi_x))= 0
\end{equation}
where $M=c_0 m_1\cdots m_{k-1}$, for some $\xi_x\in \R$ that satisfies
$0\leq \xi_x\leq m_1+\ldots+m_{k-1}$.  Moreover, since
$a^{(k)}(x)\prec x$ and $\xi_x$ is bounded, we get by
Lemma~\ref{L:basic} that
 $$
 a^{(k)}(x+\xi_x)-a^{(k)}(x)\to 0.
 $$
 Combining this with \eqref{E:B'} we get
 \begin{equation}\label{E:A=a}
   \lim_{x\to\infty}(B'(x)-M a^{(k)}(x))= 0.
 \end{equation}

 Using \eqref{E:A2} and arguing as in the proof of Lemma~\ref{L:equi1}
 we get that the sequence $B(n)$ is equidistributed in $\T$ with
 respect to the sequence of intervals $J_m$ that are chosen as
 follows:

 {\bf Case A.} If $B(x)\prec x^{1+\delta}$ for every $\delta>0$,
 then we can  set $J_m=I_m$
 (because of \eqref{E:A=a} it is easy to verify that
 this choice works). Therefore,  in this case we are immediately done.

 {\bf Case B.}  If $ B(x)\succ x^{1+\delta}$ for some $\delta>0$, we
 can choose
$$
J_m=\{n\colon r m+e(n) \leq B'(n)\leq r\cdot (m+\varepsilon)+e(n)\}
$$
where $r$ is any positive real number and $e(n)$ is any sequence that
converges to $0$.  Our objective is to choose $r$ and $e(n)$ so that
$J_m=I_m$.  We choose $r=M$ and $e(n)=B'(n)-Mb^{(k)}(n)$ (which
converges to $0$ by \eqref{E:A=a}). In this case we have that $$
J_m=\{n\colon Mm\leq Mb^{(k)}(n)\leq M(m+\varepsilon)\}=I_m.
 $$
Therefore,  in both cases we get the required equidistribution property. This completes the proof.
 %%$$
 %%\lim_{m\to\infty}\frac{|J_m \bigtriangleup I_m|}
 %%$$
\end{proof}

\subsection{Finding the polynomial patterns}
We will now complete the proof of Proposition~\ref{P:polypattern}.
First we use Proposition~\ref{L:equi2} to derive some more useful results for our particular setup.
\begin{lemma} \label{L:medium2} Suppose that $a\in \H$ is eventually
  positive and satisfies $x^k \prec a(x)\prec x^{k+1}$ for some $k\in
  \N$. Let   $r,m\in\N$ and $\varepsilon>0$.

  Then  there exist
  $n_{r,m}\in\N$ with $n_{r,m}\to\infty$ $($as $m\to\infty$ and $r$ is
  fixed$)$, such that  for all large $m$ we have
  \begin{gather*}
    \Big[\frac{a^{(k)}(n_{r,m})}{k!}\Big]=rm, \quad
    \Big[\frac{a^{(i)}(n_{r,m})}{i!}\Big]\equiv 0 \!\!\pmod{r},\
    i=0,\ldots,k-1, \\ \Big\{\frac{a^{(i)}(n_{r,m})}{i!}\Big\}\leq
    \varepsilon, \ i=0,\ldots,k.
  \end{gather*}
\end{lemma}
\begin{proof}
  Let $0<\varepsilon<1$ and $r\in \N$ be fixed. By Proposition~\ref{L:equi2} the sequence
$$
b(n)=\Big(\frac{a(n)}{r},\frac{a'(n)}{r\cdot1!},\ldots,\frac{a^{(k-1)}(n)}{r\cdot
  (k-1)!}\Big)
$$ is equidistributed in $\T^{k-1}$
with respect to the intervals
$$
I_m=\{n\in\N\colon rmk!\leq a^{(k)}(n)\leq rmk!+\varepsilon\}.
$$
Hence, for large enough $m$ there exists an $n_{r,m}\in\N$ such that
$b(n_{r,m})\in [0,\frac{\varepsilon}{r}]^{k-1}$. The result follows
by noticing that $\{x/r\}<1/r$ implies that $[x]\equiv 0
\!\!\pmod{r}$, and the estimate $\{x\}\leq r\{x/r\}$.
\end{proof}

\begin{lemma} \label{L:strong2}Suppose that $a\in \H$ is eventually
  positive and satisfies $x^k \prec a(x)\prec x^{k+1}$ for some $k\in
  \N$.

   Then there exist $\varepsilon_{r,m}\in\R,n_{r,m}\in\N$, with
  $\varepsilon_{r,m}\to 0,\ n_{r,m}\to\infty$ $($as $m\to\infty$ and
  $r$ is fixed$)$, such that for all large $m$ we have
  \begin{gather*}
    \Big[\frac{a^{(k)}(n_{r,m})}{k!}\Big]=rm, \quad
    \Big[\frac{a^{(i)}(n_{r,m})}{i!}\Big]\equiv 0 \!\!\pmod{r},\
    i=0,\ldots,k-1, \\ \Big\{\frac{a^{(i)}(n_{r,m})}{i!}\Big\}\leq
    \varepsilon_{r,m}, \ i=0,\ldots,k.
  \end{gather*}
\end{lemma}
\begin{proof}
  Let $r\in\N$. For $\varepsilon=1/k$ there exist $M_{r,k}$ and
  $n_{r,m}(k)$ such that the conclusion of Lemma~\ref{L:medium2} is
  satisfied for every $m\geq M_{r,k}$.  For $M_{r,k}\leq m< M_{r,k+1}$,
  let $\varepsilon_{r,m}=1/k$ and $n_{r,m}=n_{r,m}(k)$. Thus defined,
  the sequences $\varepsilon_{r,m}, n_{r,m}$  satisfy the conclusions of
  our lemma for every $m\geq M_{r,1}$.
\end{proof}

We are now ready to prove Proposition~\ref{P:polypattern}.
\begin{proof}[Proof of Proposition~\ref{P:polypattern}]
If $k=0$ the result follows immediately from Lemma~\ref{L:basic}.
Suppose that $k\geq 1$ and
  let $n_{r,m}$ be as in the statement of Lemma~\ref{L:strong2}.
  Using the Taylor expansion of $a(x)$ around the point $x=n_{r,m}$ we get for $n\in\N$ that
  \begin{equation}\label{E:MV2}
    a(n_{r,m}+n)=a(n_{r,m})+na'(n_{r,m})+\ldots+\frac{n^k}{k!}a^{(k)}(n_{r,m})+
    \frac{n^{k+1}}{(k+1)!}a^{(k+1)}(\xi_{r,n})
  \end{equation}
  for some $\xi_{r,n}\in [n_{r,m},n_{r,m}+n]$. By Lemma~\ref{L:basic}
  we have $a^{(k+1)}(x)\to 0$ as $x\to\infty$ (also $a^{(k+1)}(x)>0$).
Furthermore,    we have  that
  $\Big\{\frac{a^{(i)}(n_{r,m})}{i!}\Big\}\leq \varepsilon_{r,m}$ for
  $ i=0,\ldots,k$, where $ \varepsilon_{r,m}\to 0$ as $m\to\infty$.
  It follows that there exist integers $N_{r,m}$ such that
  $N_{r,m}\to\infty$ and for every $1\leq n\leq N_{r,m}$ and all large
   $m\in \N$ we have
$$
\{a(n_{r,m})\}+n\{a'(n_{r,m})\}+\ldots+n^k\Big\{\frac{a^{(k)}(n_{r,m})}{k!}\Big\}+
n^{k+1}\Big\{\frac{a^{(k+1)}(\xi_{r,n})}{(k+1)!}\Big\}
%% \leq \leq \frac{1}{l_m}+\frac{n}{l_m}+\ldots+
%% \frac{n^{k-1}}{(k-1)!l_m}
\leq \frac{1}{2}.
$$
For these values of $m$ and $n$, equation \eqref{E:MV2} gives
$$
[a(n_{r,m}+n)]=[a(n_{r,m})]+n[a'(n_{r,m})]+\ldots+n^k\Big[\frac{a^{(k)}(n_{r,m})}{k!}\Big].
$$
Remembering that $n_{r,m}$ was chosen to also satisfy
$\Big[\frac{a^{(k)}(n_{r,m})}{k!}\Big]=rm$ and $\
\Big[\frac{a^{(i)}(n_{r,m})}{i!}\Big]\equiv 0 \!\!\pmod{r}$,
$i=0,\ldots,k-1$, we get for $1\leq n\leq N_{r,m}$ that
$$
[a(n_{r,m}+n)]= r(c_{0,r,m}+c_{1,r,m}n+\ldots+c_{k-1,r,m}n^{k-1}+mn^k)
$$
for some $c_{i,r,m}\in\N$, $i=0,\ldots,k-1$. This proves the
advertised result.
\end{proof}

\section{Multiple recurrence for  Hardy field sequences}\label{S:mainpart}
In this section we shall prove our main result which we now recall:
\begin{theorem}
  Let $a\in\H$ satisfy $x^k\prec a(x) \prec x^{k+1}$ for some non-negative
  integer $k$.

  Then $S=\{[a(1)],[a(2)],\ldots\}$ is a set of multiple recurrence.
\end{theorem}

Using Proposition~\ref{P:polypattern} we see  that for elements of  $\H$ that are
eventually positive Theorem~A' is an
immediate consequence of the following result (the case of eventually negative elements of $\H$
can be treated similarly):
%%\footnote{The case where $a(x)$ is
%%  eventually negative can be treated by applying the implied multiple
%%  recurrence property of Theorem~A' to the function $-a+1$
%%  for the transformation $T^{-1}$ (notice that -[-a+1]=a if a is not an integer).}:
\begin{theorem}\label{T:recurrencepoly}
  Suppose that for every $r,m\in\N$ the set $S\subset \N$ contains
  patterns of the form
$$
\{r(mn^k+p_{r,m}(n)), 1\leq n\leq N_{r,m}\}
$$
where $p_{r,m}(n)$ is an integer polynomial of degree at most $k-1$,
and $N_{r,m}\in\N$ satisfy $N_{r,m}\to\infty$ $($as $m\to\infty$ and $r$
is fixed$)$.

Then $S$ is a set of multiple recurrence.
\end{theorem}
The rest of this section is devoted to the proof of
Theorem~\ref{T:recurrencepoly}.

\subsection{Proof of Theorem~\ref{T:recurrencepoly}  modulo a technical result}\label{SS:mainargument}
Our argument is similar to the one used to prove Theorem~B' and is
carried out in two steps. We first show that it suffices to verify a certain
multiple recurrence property for nilsystems, and we then verify this property
using an equidistribution result on nilmanifolds.

The reduction to nilsystems step is a direct consequence of Theorem~\ref{T:HoKra} and the
 following result
which serves as a substitute for
Proposition~\ref{P:seminorm'}:
\begin{proposition}\label{P:seminorm2}
  Let $(X,\mathcal{B},\mu,T)$ be a system and
  $f_0,f_1,\ldots,f_\ell\in L^\infty(\mu)$ be such that  $f_i \bot \mathcal{Z}$ for
  some $i=0,1,\ldots,\ell$, where $\mathcal{Z}$  is
  the nilfactor of the system. For fixed $k,r\in\N$ consider the polynomials
  $p_m(n)=r (mn^k+q_m(n))$ where $\deg(q_m)\leq k-1$, $m\in\N$, and let $(N_m)_{m\in\N}$
  be a sequence of positive integers with $N_m\to \infty$.

   Then
  \begin{equation}\label{E:AM}
 \lim_{M\to \infty} \frac{1}{M}\sum_{m=1}^M \Big|\frac{1}{N_m}\sum_{n=1}^{N_m} \int f_0\cdot  T^{p_m(n)} f_1
    \cdot T^{2p_m(n)}f_2\cdot\ldots\cdot T^{\ell p_m(n)}f_\ell\ d\mu \Big|=0.
  \end{equation}
\end{proposition}
The verification of the multiple recurrence property for nilsystems is
based on the following equidistribution result which serves as a substitute for
 Proposition~\ref{L:equidistribution}:
\begin{proposition}\label{P:equidistributiongeneral}
  Let $X=G/\Gamma$ be a connected  nilmanifold
  and let  $a\in G$ be an ergodic nilrotation. Suppose that
   $(q_m)_{m\in \N}$ is a sequence of integer polynomials with degree at
  most $k-1$, and  let $(N_m)_{m\in\N}$
  be a sequence of positive integers with $N_m\to \infty$.

  Then for every $F\in C(X)$ we have
$$
\lim_{M\to\infty} \frac{1}{M}\sum_{m=1}^M \Big(\frac{1}{N_m}\sum_{n=1}^{N_m}
F(a^{mn^k+q_m(n)}\Gamma)\Big)= \int F \ dm_X.
$$
\end{proposition}
\begin{proof}
  The argument is identical to the one used to prove
  Proposition~\ref{L:equidistribution}.
\end{proof}

\begin{proof}[Proof of Theorem~\ref{T:recurrencepoly}]
  Using Proposition~\ref{P:seminorm2} and
  Proposition~\ref{P:equidistributiongeneral} the argument is
  identical to the one used to finish the proof of Theorem~B'
  (Section~\ref{SS:warmupconclusion}).
\end{proof}
\subsection{Proof of Proposition~\ref{P:seminorm2}}
%%Since the proof of Proposition~\ref{P:seminorm2} is rather technical, the
%%reader is advised to look first at the proof of  Lemma~\ref{L:seminorm}
%% which explains  our argument in its simplest possible setting (see also
%%Example~\ref{Ex:vdc}).
The
proof  of Proposition~\ref{P:seminorm2} is carried out in two steps.
First we establish an estimate
about general polynomial families using an inductive argument that is
frequently used when one deals with multiple ergodic averages along
polynomial iterates. We then apply this estimate to show that the
ergodic averages we are interested in are majorized by some
multi-parameter polynomial ergodic averages. A known result  shows
that these averages are controlled by nilsystems, so the same should
be the case for our averages.
\subsubsection{A PET induction argument}
We start with some notational conventions that we  use henceforth: If
$h=(h_1,\ldots,h_r)$, $H=(H_1,\ldots,H_r)$, when we write $1\leq h\leq
H$ we mean $1\leq h_i\leq H_i$ for $i=1,\ldots,r$, and when we write
$|H|$ we mean $H_1\cdot\ldots\cdot H_r$. With $o_{h}(1)$ we denote an
expression that converges to zero when $h_1,\ldots,h_r\to\infty$, and
for $N\in\N$ we denote by $o_{h, N, h\prec N}(1)$ a quantity that goes
to $0$ if $h_i,N\to\infty$ and $h_i/N\to 0$.

We briefly review some notions from \cite{HK2} (most of which where
introduced in \cite{Be2}).  We say that a property holds for
\emph{almost every $h\in\Z^r$} if it holds outside of  a subset of
$\Z^r$ with upper density zero. We remark that the set of zeros of any
non-identically zero polynomial $p\colon \Z^r \to \Z$ has zero upper
density.
If $a_0,a_1, \ldots, a_k\colon \Z^r\to \Z$ are integer polynomials, with $a_k$ not identically zero,  we call a
function $p\colon \Z^{r+1}\to \Z$ defined by
$p(h,n)=a_k(h)n^k+\cdots+a_1(h)n+a_0(h)$ an \emph{integer polynomial
  with $r$ parameters and degree $k$}. If $p(h,n)$ has
  degree $k$, then for almost every $h\in \Z^r$, the degree of
  the polynomial $p(h,n)$, with respect to the variable $n$ is $k$.
  A set
$\mathcal{P}=\{p_1(h,n),\ldots,p_k(h,n)\}$, where $p_i(h,n)$ are
integer polynomials with $r$ parameters, is called a \emph{family of
  integer polynomials with $r$ parameters}. The polynomials in $\mathcal{P}$ are
  \emph{non-constant} if they all have positive degree,  and \emph{essentially
  distinct} if  all their pairwise differences   have positive degree.
The maximum degree of the polynomials  is called the \emph{degree} of the polynomial family and
is denoted by $\text{deg}(\mathcal{P})$.  Given a
polynomial family $\mathcal{P}$ with several parameters, let $\mathcal{P}_i$ be the subfamily
of polynomials of degree $i$ in $\mathcal{P}$. We let $w_i$ denote the
number of distinct leading coefficients that appear in the family
$\mathcal{P}_i$. The vector $(d,w_d,\ldots,w_1)$ is called the
\emph{type} of the polynomial family $\mathcal{P}$.

We shall use an induction scheme, often called PET induction
(Polynomial Exhaustion Technique), on types of polynomial families
that was introduced by Bergelson in \cite{Be2}.  To do this we order
the set of all possible types lexicographically, this means,
$(d,w_d,\ldots,w_1)>(d',w_d',\ldots,w_1')$ if and only if in the first
instance where the two vectors disagree the coordinate of the first
vector is greater than the coordinate of the second vector.

\begin{lemma}\label{L:tolinear}
  Let $\mathcal{P}=\{p_1,\ldots,p_k\}$ be a family of non-constant essentially
  distinct polynomials with $r$ parameters such that $\text{deg}(\mathcal{P})=\text{deg}(p_1)$,
    $(X,\mathcal{B},\mu,T)$  be a system, and
   $f_1,\ldots,f_k\in L^\infty(\mu)$ with
  $\norm{f_i}_{L^\infty(\mu)}\leq 1$ for $i=1,\ldots, k$.

  Then  there exist $s\in\N$,
  depending only on the type of $\mathcal{P}$,  and
   $\tilde{r}\in \N$ depending only on $r$ and the type of $\mathcal{P}$,  a family of non-constant
   essentially distinct
  linear polynomials $\{q_1,\ldots, q_s\}$ with  $r+\tilde{r}$ parameters,
  and functions $g_1,\ldots, g_s \in L^\infty(\mu)$
  $($independent of $h$ and $\tilde{h}$$)$ with $\norm{g_i}_{L^\infty(\mu)}\leq 1$
  and $g_1=f_1$,  such that for every $h\in \Z^r$ we have
  \begin{multline}\label{E:CSVDC}
    \norm{\frac{1}{N}\sum_{n=1}^N T^{p_1(h,n)}f_1\cdot
      \ldots\cdot
    T^{p_k(h,n)}f_k}_{L^2(\mu)}^{2^{\tilde{r}}}\ll \\
   \frac{1}{|\tilde{H}|}\sum_{1\leq \tilde{h}\leq \tilde{H}}\norm{\frac{1}{N}
    \sum_{n=1}^N  T^{q_1(\tilde{h},h, n)}g_1\cdot \ldots\cdot
      T^{q_s(\tilde{h},h,n)}g_s}_{L^2(\mu)} +o_{N,\tilde{H}, \tilde{H}\prec N}(1)
  \end{multline}
  where $\tilde{H}=(H_1,\ldots,H_{\tilde{r}})$, and the implied constant depends only on  the type of
  the polynomial family  $\mathcal{P}$.
  %% , and by $o_{N,H, H\prec N}(1)$ we denote a quantity that goes to
  %% $0$ as $H,N\to\infty$ in a way that %%$H/N\to 0$.
\end{lemma}

\begin{proof}
We remark that  throughout the proof all the implied constants depend only on the type of
  the polynomial family  $\mathcal{P}$.

  We shall use  induction on the type of the polynomial family $\mathcal{P}$.
  Assume that the statement holds for all polynomial families with several parameters
   and
  type less than $(d,w_d,\ldots,w_1)$, and suppose that $\{p_1,\ldots,p_k\}$ is
  a polynomial family with $r$ parameters and type $(d,w_d,\ldots,w_1)$.
   Let  $d_0$ be
  the first positive integer for which $w_{d_0}\neq 0$.
  %% so the
  %%polynomial family has type $(d,w_d,\ldots,w_{d_0},0\ldots,0)$.
  Without loss of generality we can assume that the polynomial $p_k$
  has minimal degree and is such that  the polynomials
 $p_i(h,n)-p_k(h,n)$, $p_j(h,n+h_{1})-p_k(h,n)$,
 for  $i=1,\ldots,k-1$, $j=1,\ldots,k$, have degree less than or equal
  to the
 degree of the polynomial $p_1(h,n)-p_k(h,n)$.
 Furthermore, we can assume that $\deg(p_1)\geq
  2$, otherwise all polynomials are already linear (since $p_1$ has
  maximal degree), in which case there is nothing to prove. In order to
  carry out the inductive step we consider two cases.

  {\bf Case 1.} Suppose that $\deg(p_k)=d_0\geq 2$.
  %% Using the Cauchy Schwarz inequality and the estimate
  %% $\norm{f_0}_{L^\infty(\mu)}\leq 1$ we get
%%$$
%%\Big|\E_{1\leq n\leq N}\int f_0 \cdot T^{p_1(n)}f_1\cdot \ldots\cdot
%% T^{p_k(n)}f_k \ d\mu \Big|^{2}\leq\norm{\E_{1\leq n\leq N}
%%   T^{p_1(n)}f_1\cdot \ldots\cdot T^{p_k(n)}f_k}_{L^2(\mu)}^{2}
%%$$
Let  $\tilde{r}\in \N$  be the integer that is determined by the induction hypothesis. We are going to estimate
\begin{equation}\label{E:lkmn}
\norm{\frac{1}{N}\sum_{n=1}^N T^{p_1(h,n)}f_1\cdot
      \ldots\cdot
      T^{p_k(h,n)}f_k}_{L^2(\mu)}^{2^{\tilde{r}+1}}.
      \end{equation}
  We use Lemma~\ref{L:VDC}, then  factor out the measure preserving transformation
  $T^{p_k(h,n)}$ from the resulting integrals, and use the Cauchy
  Schwarz inequality. Since the sup norm of all
  functions is bounded by $1$, we find that for  every  $h\in \Z^r$ the expression in
  \eqref{E:lkmn} is bounded by some constant times
  %% \begin{multline}\label{E:CSVDC'}
  %%   \E_{1\leq h_i\leq H_i}\Big|\E_{1\leq n\leq N}\int f_0 \cdot
  %%   T^{p_1(h_1,\ldots,h_r,n)}f_1\cdot %%\ldots\cdot
  %%   T^{p_k(h_1,\ldots,h_r,n)}f_k \ d\mu \Big|^{2^{l_1}}\leq\\
  \begin{equation}\label{E:CSVDC'}
    \frac{1}{H_1}\sum_{h_1=1}^{H_1}\norm{\frac{1}{N}\sum_{n=1}^N T^{\tilde{p}_1(h_1,h,n)}\tilde{f}_1\cdot \ldots\cdot
      T^{\tilde{p}_{l}(h_1,h,n)}\tilde{f}_l}_{L^2(\mu)}^{2^{\tilde{r}}} +o_{N,H_1, H_1\prec N}(1),
  \end{equation}
  where $l=2k-1$, $\tilde{f}_1=f_1$, $\tilde{p}_1(h_{1},h,n)= p_1(h,n)-p_k(h,n)$,
  the functions $\tilde{f}_2,\ldots, \tilde{f}_{l}$ are bounded by $1$ and do not depend on the parameters
  $h_{1}$, $h$, and the polynomials $\{\tilde{p}_2,\ldots,\tilde{p}_l\}$ have the form
$$
p_i(h,n)-p_k(h,n), \ \text{ or } \ p_j(h,n+h_{1})-p_k(h,n)
$$
for some  $i=2,\ldots,k-1$ and  $j=1,\ldots,k$.
%% \begin{gather*}
%%   q_1=p_1(h_1,\ldots,h_r,n)-p_k(h_1,\ldots,h_r,n),\ldots,q_{k-1}=p_{k-1}(h_1,\ldots,h_r,n)-p_k(h_1,\ldots,h_r,n),\\
%%   q_k=p_1(h_1,\ldots,h_r,n+h_{r+1})-p_k(h_1,\ldots,h_r,n),\ldots, %%q_{2k-1}=p_k(h_1,\ldots,h_r,n+h_{r+1})-p_k(h_1,\ldots,h_r,n)
%% \end{gather*}
It is easy to verify that  $\{\tilde{p}_1,\ldots,\tilde{p}_{l}\}$ is a family of
non-constant essentially distinct polynomials with  $r+1$ parameters,  type
strictly  smaller than the type of the family $\mathcal{P}$, and the
polynomial $\tilde{p}_1$  has maximal degree.  Therefore, we can apply the induction
hypothesis   to give
a bound for the norm that appears in \eqref{E:CSVDC'}.
 Putting these estimates together
we produce the desired bound, completing the inductive step.

{\bf Case 2.} Suppose that $\deg(p_k)=d_0=1$.  After possibly
rearranging the polynomials $p_2,\ldots, p_{k-1}$ we can assume that
$\deg{p_i}=1$ if and only if $i\geq k_0$, for some $k_0$ that
satisfies $2\leq k_0\leq k$.  Notice that since the polynomials $p_i$
are linear for $i\geq k_0$, for every $h_1\in\N$ we have
\begin{equation}\label{E:linear}
  p_k(n+h_{1})-p_k(n)=p_k(h_1), \ p_i(n+h_{1})-p_k(n)=p_i(n)-p_k(n)+p_i(h_1)
  \text{ for } i=k_0,\ldots,k-1.
\end{equation}

Arguing as in Case $1$ and keeping in mind the identities
\eqref{E:linear} we get the estimate
\begin{multline}\label{E:bvc}
  \norm{\frac{1}{N}\sum_{n=1}^N T^{p_1(h,n)}f_1\cdot \ldots\cdot
    T^{p_k(h,n)}f_k}_{L^2(\mu)}^2\leq\\ \frac{4}{H_1}
    \sum_{h_1=1}^{H_1}\norm{\frac{1}{N}\sum_{n=1}^N T^{\tilde{p}_1(h_1,h,n)}\tilde{f}_1\cdot \ldots\cdot
    T^{\tilde{p}_{l}(h_1,h,n)}\tilde{f}_{l}}_{L^2(\mu)} +o_{N,H_1, H_1\prec N}(1),
\end{multline}
where  $l=k-k_0-2$, $\tilde{f}_1=f_1$,
$\tilde{p}_1(h_{1},h,n)= p_1(h,n)-p_k(h,n)$, and the polynomials $\tilde{p}_2,\ldots,\tilde{p}_l$
have the form
$$
p_i(h,n)-p_k(h,n), \ \text{ or } \ p_j(h,n+h_{1})-p_k(h,n)
$$
for some $i=2,\ldots,k-1$ and $j=1,\ldots,k_0-1$.  It is easy to verify that
 $\{\tilde{p}_1,\ldots,\tilde{p}_{l}\}$ is a family of
non-constant essentially distinct polynomials with  $r+1$ parameters  and type
strictly  smaller than the type of the family $\mathcal{P}$.

Note that in this case some of the functions $\tilde{f}_i$ may depend on
$h_{1}$, but this can happen only for those indices $i$ for which
$\deg(p_i)=1$. In order to get rid of these functions  we
use Lemma~\ref{L:VDC} again to get a bound for the expression
$$
\norm{\frac{1}{N}\sum_{n=1}^N T^{\tilde{p}_1(h_1,h,n)}\tilde{f}_1\cdot
  \ldots\cdot T^{\tilde{p}_{l}(h_1,h,n)}\tilde{f}_{l}}_{L^2(\mu)}
$$
that involves one less linear term.  After repeating this step a
finite number of times ($w_1-1$ in total), we eventually get an expression without any linear terms
(see Example~\ref{Ex:vdc}).
Combining this with the estimate \eqref{E:bvc} we get
\begin{multline}
  \norm{\frac{1}{N}\sum_{n=1}^N T^{p_1(h,n)}f_1\cdot \ldots\cdot
    T^{p_k(h,n)}f_k}_{L^2(\mu)}^{2^{w_1}}\ll\\
  \frac{1}{|\tilde{H}|}\sum_{1\leq \tilde{h}\leq \tilde{H}}\norm{\frac{1}{N}\sum_{n=1}^N
   T^{\tilde{p}_1(\tilde{h},h,n)}\tilde{f}_1\cdot \ldots\cdot
    T^{\tilde{p}_{l_1}(\tilde{h},h,n)}\tilde{f}_{l_1}}_{L^2(\mu)} +o_{N,\tilde{H}, \tilde{H}\prec N}(1)
\end{multline}
for some $l_1\in\N$, where $\tilde{h}=(h_1,\ldots,h_{w_1})$, $\tilde{H}=
(H_1,\ldots,H_{w_1})$.  The  family of non-constant essentially
distinct polynomials $\{\tilde{p}_1,\ldots,\tilde{p}_{l_1}\}$ has $r+w_1$ parameters, type
strictly smaller than the type of the family $\mathcal{P}$,
%%$(d,w_d,\ldots,w_2,0)$,
and the functions
$\tilde{f}_1,\ldots,\tilde{f}_{l_1}$ are bounded by $1$ and do not depend on the parameters $h$, $\tilde{h}$.  We can
now use the induction hypothesis, as in Case $1$, to carry out the
inductive step and complete the proof.
\end{proof}

We illustrate the method used in the  previous proof with the following example:
\begin{example}\label{Ex:vdc}
We start with polynomial family $\{n,2n,n^2\}$ that has  type $(2,1,2)$  and study the corresponding
multiple ergodic averages
$$
\frac{1}{N}\sum_{n=1}^N T^{n^2}f_1\cdot T^{2n}f_2\cdot T^n f_3.
$$
Since this expression involves linear terms, we  perform the operation described in Case 2.
We are led to study an
 average over $h_1$ of a power of the  $L^2$ norms of the  averages
$$
%%\frac{1}{H_1}\sum_{1\leq h_1\leq H_1}
\frac{1}{N}
\sum_{n=1}^N T^{(n+h_1)^2-n}\overline{f}_1
\cdot T^{n^2-n}f_1\cdot T^{n}(T^{2h_1}\overline{f}_2 \cdot f_2),
$$
which involves a polynomial family with $1$ parameter that has  type $(2,1,1)$.
We perform one more time the operation described in Case 2. We are led to study an
average over $h_1,h_2$ of a power of the $L^2$ norms of the averages
$$
\frac{1}{N}\sum_{n=1}^N T^{(n+h_1+h_2)^2-2n-h_2}f_1
\cdot T^{(n+h_1)^2-2n}\overline{f}_1\cdot T^{(n+h_2)^2-2n-h_2}\overline{f}_1
\cdot T^{n^2-2n}f_1,
$$
which involves a polynomial family with $2$ parameters that has  type $(2,1,0)$.
Since the resulting expression has no linear terms,  we  perform the operation described in Case 1. We are led to study an
average over $h_1,h_2, h_3$ of the $L^2$ norms of the averages
\begin{align*}
\frac{1}{N}\sum_{n=1}^N
&T^{2(h_1+h_2+h_3)n+(h_1+h_2+h_3)^2-h_2-2h_3}\overline{f}_1
\cdot T^{2(h_1+h_2)n+(h_1+h_2)^2-h_2}f_1\cdot \\
& T^{2(h_1+h_3)n +(h_1+h_3)^2-2h_3}f_1\cdot
T^{2h_1n+h_1^2}\overline{f}_1 \cdot
T^{2(h_2+h_3)n+(h_2+h_3)^2-h_2-2h_3}f_1\cdot
  \\
 &\qquad \qquad \qquad \qquad \qquad \qquad \qquad \qquad \qquad \qquad \qquad T^{2h_2n+h_2^2-h_2}\overline{f}_1\cdot
T^{2h_3n+h_3^2-2h_3}\overline{f}_1,
\end{align*}
which involves a family of linear polynomials with $3$ parameters that has type $(1,7)$.
\end{example}
Next we use  the previous lemma to  bound  some multiple ergodic averages involving a
collection of polynomials of a single variable
with some multiple ergodic averages involving a collection of polynomials of several variables
that have some convenient special form.
\begin{lemma} \label{L:VDCmain} Let $\mathcal{P}=\{p_1,\ldots,p_k\}$ be a family
  of non-constant essentially distinct polynomials, $(X,\mathcal{B},\mu,T)$ be a
  system, and $f_0,f_1, \ldots, f_k\in L^\infty(\mu)$ with
  $\norm{f_i}_{L^\infty(\mu)}\leq 1$.

  Then there exist $r,s \in \N$, non-constant
  essentially distinct polynomials $P_1,\ldots,P_s\colon \Z^r\to \Z$
  that are  independent of $n$ and each $P_i$ is an integer combination of polynomials of the form
  $p_i(n+\sum_{j\in J}h_j)$ where $J$ is some subset (possibly empty) of
  $\{1,\ldots,r\}$,
  and functions $F_0,F_1,\ldots,F_s\in L^\infty(\mu)$, such that
  $F_1=f_1$, and
  \begin{multline*}
    \Big| \frac{1}{N}\sum_{n=1}^N \int f_0\cdot
    T^{p_1(n)}f_1\cdot \ldots\cdot
    T^{p_k(n)}f_k \ d\mu \Big|^{2^r}\ll\\
    \frac{1}{|H|}\sum_{1\leq h\leq H}\Big|\int F_0 \cdot
    T^{P_1(h)}F_1\cdot \ldots\cdot T^{P_s(h)}F_s \ d\mu \Big| +o_{N,H,
      H\prec N}(1)
  \end{multline*}
  where $H=(H_1,\ldots,H_r)$ and the implied constant depends only on the type of
  the polynomial family  $\mathcal{P}$.
\end{lemma}
\begin{remark}
 Our assumptions force the polynomials $P_1,\ldots,P_s$
to have very special form, we are going to take advantage of this  property  later.
\end{remark}
\begin{proof}
We remark that  throughout the proof all the implied constants depend only on the type of
  the polynomial family  $\mathcal{P}$.

We are going to estimate the quantity
$$
\Big| \frac{1}{N}\sum_{n=1}^N \int f_0\cdot
    T^{p_1(n)}f_1\cdot \ldots\cdot
    T^{p_k(n)}f_k \ d\mu \Big|^{2^r}
$$
for some appropriate choice of $r$.
 We can assume that the polynomial $p_1$ has maximal degree. Indeed, if this is not the case,
we can  factor out   the measure preserving transformation $T^{p_{i_o}(n)}$ where
  $p_{i_o}$ is some polynomial of maximal degree and work with the resulting family.

Using  the Cauchy Schwarz inequality  and  Lemma~\ref{L:tolinear}
we get that   there exist $r_1,s_1\in\N$, depending
  only on the type of $\mathcal{P}$, a family of
  non-constant essentially distinct linear polynomials $q_1,\ldots, q_{s_1}$ with
  $r_1$ parameters,
  and functions $\tilde{f}_1,\ldots,\tilde{f}_{s_1}\in
  L^{\infty}(\mu)$, such that $\tilde{f}_1=f_1$ and
  \begin{multline}\label{E:CSVDC''}
    \Big| \frac{1}{N}\sum_{n=1}^N \int f_0\cdot
    T^{p_1(n)}f_1\cdot \ldots\cdot
    T^{p_k(n)}f_k \ d\mu \Big|^{2^{r_1}}\ll\\
    \frac{1}{|H|}\sum_{1\leq h\leq H} \norm{\frac{1}{N}\sum_{n=1}^N
     T^{q_1(h,n)}\tilde{f}_1\cdot \ldots\cdot
      T^{q_{s_1}(h,n)}\tilde{f}_{s_1}}_{L^2(\mu)} +o_{N,H, H\prec N}(1)
  \end{multline}
  where $H=(H_1,\ldots,H_{r_1})$. Since all the polynomials are
  linear, arguing as in the proof of Lemma~\ref{L:seminorm} we get
  that there exist $r_2,s_2 \in \N$, non-constant essentially distinct polynomials
  $P_1,\ldots,P_{s_2}\colon \Z^{r_2}\to\Z$, and functions
  $F_0,\ldots,F_{s_2}\in L^\infty(\mu)$, such that $F_1=\tilde{f}_1=f_1$, and
  \begin{multline}\label{E:CSVDC'''}
    \norm{\frac{1}{N}\sum_{n=1}^N T^{q_1(h,n)}\tilde{f}_1\cdot
      \ldots\cdot
      T^{q_{s_1}(h,n)}\tilde{f}_{s_1}}_{L^2(\mu)}^{2^{r_2}}\ll\\
    \frac{1}{|\tilde{H}|}\sum_{1\leq \tilde{h}\leq \tilde{H}}\Big|\int F_0 \cdot
    T^{P_1(\tilde{h},h)}F_1\cdot \ldots\cdot T^{P_{s_2}(\tilde{h},h)}F_{s_2} \ d\mu \Big|
    +o_{N,\tilde{H}, \tilde{H}\prec N}(1)
  \end{multline}
  where $\tilde{H}=(\tilde{H}_1,\ldots,\tilde{H}_{r_2})$. Combining \eqref{E:CSVDC''} and
  \eqref{E:CSVDC'''}, and noticing that $\{P_1(\tilde{h},h),\ldots
  P_{s_2}(\tilde{h},h)\}$ is a family of non-constant essentially distinct polynomials
  $\Z^{r_1+r_2}\to \Z$ we get the advertised estimate with
  $r=r_1+r_2$. Furthermore, looking at the proof of
  Lemma~\ref{L:tolinear} and Lemma~\ref{L:seminorm}, we see that the
  polynomials $P_i$ are constructed starting from the family
  $\{p_1,\ldots,p_k\}$ and performing  $r$ times one of the
  following two operations: $(i)$ form the polynomial $p(n)-q(n)$
  or $(ii)$
 form the polynomial  $p(n+h_i)-q(n)$, where $p$ and $q$ are already defined polynomials and $i\in \{1,\ldots,r\}$.  It follows that the polynomials $P_i$
  have the advertised form. This completes the proof.
\end{proof}

\subsubsection{Conclusion of the reduction} We shall now use
Lemma~\ref{L:VDCmain} to show that the nilfactor is characteristic for
the multiple ergodic averages related to the multiple recurrence
problem of Theorem~\ref{T:recurrencepoly}. We first need a simple
lemma:
\begin{lemma}\label{L:leading} Fix $k\in\N$. Let $(p_m)_{m\in\N}$ be a
  sequence of integer polynomials of the form $ p_m(n)=mn^k+q_m(n)$
  where $\deg(q_m)\leq k-1$ and for $l_i, h_i\in\Z$ let
$$
P_{m,n_1,\ldots,n_r}(n)=\sum_{i=1}^r l_i \ \! p_m(n+h_i).
$$

Then the leading coefficient of  $P_{m,h_1,\ldots,h_r}(n)$ has the form
$m \cdot P(h_1,\ldots, h_r)$ for some polynomial $P\colon \Z^r\to \Z$.
\end{lemma}
\begin{proof}
  %% Let $P_j(t)=\sum_{i=1}^r l_i \ \! (t+n_i)^j$.  Since
  %% $P^{(k-j)}(t)=(k-j)!P_j(t)$has degree greater than %%the degree of the polynomials $ \sum_{i=1}^r l_i \ \! (n+n_i)^j$ for all $0\leq j<k$
  For every choice of integers $l_i,n_i$, and positive integer $j$
  with $j<k$, the degree of the polynomial $\sum_{i=1}^r l_i \ \!
  (t+h_i)^k$ is greater than the degree of the polynomial
  $\sum_{i=1}^r l_i \ \! (t+h_i)^j$ (up to a constant, we get the
  second polynomial by differentiating the first several times), as
  long as the two polynomials are not identically zero.  Hence, the
  leading coefficient of the polynomial $P_{m,h_1,\ldots,h_r}(n)$ is
  the same as the leading coefficient of the polynomial $ m\cdot
  \sum_{i=1}^r l_i \ \! (n+h_i)^k.  $ The result follows.
\end{proof}

\begin{proof}[Proof of Proposition~\ref{P:seminorm2}]
We remark that  throughout the proof all the implied constants depend only on the type of
  the polynomial family  $\mathcal{P}$.

  Without loss of generality we can assume that
  $\norm{f_i}_{L^\infty(\mu)}\leq 1$ for $i=0,1,\ldots,\ell$, and $f_1
  \bot \mathcal{Z}$.  Let
  $$
 A_M=\frac{1}{M}\sum_{m=1}^M \Big|\frac{1}{N_m}\sum_{n=1}^{N_m} \int f_0\cdot  T^{p_m(n)} f_1
    \cdot T^{2p_m(n)}f_2\cdot\ldots\cdot T^{\ell p_m(n)}f_\ell\ d\mu\Big|
     $$
  We shall use Lemma~\ref{L:VDCmain} to get a bound for the averages
  $A_M$ that is independent of the polynomials $q_m$ (remember $p_m(n)=r(mn^k+q_m(n))$).

  Using the Cauchy-Schwarz inequality we have for every $t\in\N$ that
  \begin{equation}\label{E:CS}
    |A_M|^{2^t} \leq \frac{1}{M}\sum_{m=1}^{M}
      \Big| \frac{1}{N_m}\sum_{n=1}^{N_m} \int f_0\cdot  T^{p_m(n)}f_1
      \cdot T^{2p_m(n)}f_2\cdot\ldots\cdot T^{\ell p_m(n)}f_\ell\ d\mu\Big|^{2^t}.
  \end{equation}
  If $r$ is the integer given by Lemma~\ref{L:VDCmain}, letting $t=r$ we have that
  there exist $s \in \N$, non-constant essentially distinct polynomials
  $P_{m,1},\ldots,P_{m,s}\colon \Z^t\to\Z$, and functions
  $F_0,\ldots,F_s\in L^\infty(\mu)$, such that $F_1=f_1$, and
  \begin{multline}\label{M:22}
    \Big| \frac{1}{N_m}\sum_{n=1}^{N_m} \int f_0 \cdot T^{p_m(n)}f_1
      \cdot T^{2p_m(n)}f_2\cdot\ldots\cdot T^{\ell p_m(n)}f_\ell \ d\mu\Big|^{2^t}\ll \\
    \frac{1}{|H_m|}\sum_{1\leq h\leq H_m}\Big|\int F_0 \cdot
    T^{P_{m,1}(h)}F_1\cdot \ldots\cdot T^{P_{m,s}(h)}F_s \ d\mu \Big|
    +o_{N_m,H_m, H_m\prec N_m}(1)
  \end{multline}
  where $H_m=(H_{m,1},\ldots,H_{m,t})$.  By Lemma~\ref{L:VDCmain}, the
  polynomials $P_{m,1},\ldots, P_{m,s}$ have the form $ \sum_{i=1}^t
  l_i \ \! p_m(n+h_i') $ where $h_i'= \sum_{j\in I_{i,m}}h_j$ for some
  subsets $I_{i,m}$ of $\{1,\ldots,t\}$.  Since the polynomials
  $P_{m,1},\ldots,P_{m,s}$ are constant in $n$, and
  $p_m(n)=r(mn^k+q_m(n))$ for some $q_m\in \Z[x]$ with $\deg{q_m}\leq
  k-1$, we get from Lemma~\ref{L:leading} that
$$
P_{m,i}(n)=rm\! \cdot P_i(h),
$$ for $i=1,\ldots,s$,  for some non-constant essentially distinct polynomials
$P_1\ldots, P_s\colon \Z^t\to \Z$. Keeping this in mind, and putting
together \eqref{E:CS} and \eqref{M:22} we find that
$$
|A_M|^{2^t} \ll\frac{1}{|\Phi_M|}\sum_{(m,h)\in
  \Phi_M}\Big|\int F_0 \cdot\ T^{rmP_1(h)}F_1\cdot \ldots \cdot\
T^{rmP_s(h)}F_s\ d\mu\Big| +o_{M,H_m, H_m\prec N_m}(1)
$$
where
$$
\Phi_{M}=\{ (m,h)\in \mathbb{N}^{t+1}\colon 1\leq m\leq M, \ 1\leq
h\leq H_m\ \}.
$$
We can  choose  $H_m$ such that $H_m/N_m\to 0$ and
$(\Phi_M)_{M\in\N}$ forms a F{\o}lner sequence of subsets of
$\N^{t+1}$.
%%(for example choose $H_m\to \infty$ such that $H_m/N_m\to
%%0$ and the sequence $H_{m+1}- H_{m}$ remains bounded).
Since
$F_1=f_1\bot \mathcal{Z}$, using  Corollary~\ref{C:L2'} we get that the last expression converges to zero
when $M\to\infty$. This shows that $A_M$ converges to zero in
$L^2(\mu)$ as $M\to \infty$ and completes the proof.
\end{proof}

%% \appendix
%% \begin{center}
%%   APPENDIX
%% \end{center}

\section{Appendix: Single recurrence for Hardy
field  sequences}\label{SS:single}
%% \subsection{Single recurrence for Hardy sequences (proof of
%%   Theorem~C')}
%% \emph{Notation}: Let $a(n), b(n)$ be two Hardy sequences. We say
%% that $a(n)$ is good for single recurrence if
%% $S=\{[a(1)],[a(2)],\ldots\}$ is a set of single recurrence. When we
%% write $[a(n)]=[b(n)]$ we mean that the equality holds for all large
%% $n\in \N$. If $a(x)\sim b(x)$ if $a(x)-b(x)\to 0$ as $x\to \infty$.
In this last section we deal with single recurrence properties of Hardy field sequences
and improve upon the single recurrence versions of Theorems~A and A'.

Sark\"ozy (\cite{Sa}), and independently Furstenberg (\cite{Fu2}),
showed that if a non-constant polynomial $q\in\Z[x]$ has zero constant term, then
the sequence $q(n)$ is good for single recurrence. More
generally, the same is true for (non-constant) sequences of the form
$[q(n)]$ where $q\in \R[x]$, with $q(0)=0$ (\cite{BH}).
 If the constant term of the polynomial $q\in \R[x]$ is non-zero, then the sequence
 $[q(n)]$ is still good for single recurrence, provided that
  $q$ is not of the form  $q=cp+d$ for some $p\in\Z[x]$ and $c,d\in \R$. In this case
  determining whether $[q(n)]$   is
good for recurrence depends on intrinsic properties of the polynomial
$q$.\footnote{It can be shown that the sequence $q(n)$, where $q\in
  \Z[x]$, is good for recurrence if and only if the set $\{q(n)\colon
  n\in\N$\} contains multiples of every positive integer (\cite{KM}),
  and the sequence $[an+b]$, $a,b\in\R$, is good for
  recurrence if and only if there exists an integer $k$ such that
  $b+ka\in [0,1]$ (\cite{BH}).} For example, the sequences
$3n+3$, $n^2-1$, $[\sqrt{5}n+1]$, $[\sqrt{5}n+3]$
are good for recurrence, but the sequences $3n+1$,
$n^2+1$, $[\sqrt{5}n+2]$,
are bad for recurrence.

We shall show that if the function $a\in \H$ has polynomial growth and stays away from
polynomials of the form $cp(x)$, where $p\in \Z[x]$ and $c\in\R$, then
the sequence $[a(n)]$ is always good for single recurrence. This is
the statement of Theorem~C' which we now repeat.
%% \footnote{More generally the sequence $([\sqrt{5}n+c])_{n\in\N}$ is
%%   bad for recurrence only if for %%every $k\in \Z$ we have $c+k\sqrt{5}\notin [0,1]$. On the other hand, the sequence %%$([\sqrt{5}n^2+c])_{n\in\N}$ is bad for recurrence if and only if for every $k\in \Z$ we have %%$c+k^2\sqrt{5}\notin [0,1]$, which happens to be the case for more values of $c$ than in the linear %%case.}

%%   \begin{theorem}\label{T:singlehardy} Let $a(x)$ be a
%%     subpolynomial Hardy field function. If $a(x)-cp(x)\to \infty $
%%     for every $p\in \Q[x]$ and $c\in \R$, then
%%     $S=\{[a(1)],[a(2)],\ldots\}$ is a set of single recurrence.

\begin{theoremC'}
  Let $a\in\H$ have  polynomial growth and suppose that  $|a(x)-cp(x)|\to \infty $ for every $p\in
  \Z[x]$ and
  $c\in \R$.

  Then $S=\{[a(1)],[a(2)],\ldots\}$ is a set of single
  recurrence.
\end{theoremC'}

We are going to use the following lemma:

\begin{lemma}\label{L:t}
  Suppose that $p\in \R[x]$ is non-constant with leading coefficient
  $\alpha$, the real number $\beta$ is such that
  $1/\beta \notin \mathbb{Q}/ \alpha+\mathbb{Q}$, and
  %% Define the sequence $(S_M)_{M\in\N}$ by
%%$$
%%S_M=\{(m,n)\in\Z^2\colon 1\leq m\leq M, \ n_m\leq n\leq n_m+N_m\}
%%$$
  $n_m,N_m\in\N$ are such that $N_m\to\infty$ as $m\to\infty$.
  %% \footnote{Notice that $(S_M)_{M\in\N}$ is not necessarily a
  %%   F{\o}lner sequence.}

  Then for every $t\in (0,1)$, and Riemann integrable function
  $\phi\colon \T\to \C$, the averages
$$
\frac{1}{M}\sum_{m=1}^M\Big(\frac{1}{N_m} \sum_{n=n_m}^{n_m+N_m}
\phi(p(n)+m\beta)\cdot e\big([p(n)+m\beta]t\big)\Big)
$$
converge to zero as $M\to\infty$.
\end{lemma}
\begin{proof}
  Suppose first that $t$ is irrational.
 We shall show that for every Riemann integrable function
 $f$ of $\T^2 $ we have
\begin{equation}\label{E:fkl}
\lim_{M\to\infty}\frac{1}{M}\sum_{m=1}^M\Big(
\frac{1}{N_m}\sum_{n=n_m}^{n_m+N_m}
f\big((p(n)+m\beta)t,p(n)+m\beta\big)\Big)=\int_{\T^2} f(x,y)\ dx  dy.
\end{equation}
Applying this for $f(x,y)=\phi(y)\cdot e(x-t\{y\})$, and noticing that
the integral of $f$ is zero, we  get the advertised result.

Next we verify \eqref{E:fkl}.
Using  a standard approximation argument by trigonometric
 polynomials we can assume that $f(x,y)=e(l_1x+l_2y)$ for some $l_1,l_2\in \Z$
 not both of them $0$. So it suffices to show that
 \begin{equation}\label{E:fkl'}
\lim_{M\to\infty}\frac{1}{M}\sum_{m=1}^M
\ e\big(m(l_1t+l_2)\beta\big)A_m=0
\end{equation}
where
$$
A_m=\frac{1}{N_m}\sum_{n=n_m}^{n_m+N_m}
e\big(l_1p(n)t+l_2p(n)\big).
$$

Suppose first that $l_1=0$. Then $l_2\neq 0$, and  since the sequence $A_m$ is known to converge and  $\beta$ is irrational, we get that  the limit in \eqref{E:fkl'} is $0$.

Suppose now that $l_1\neq 0$. Notice that the leading coefficient of the polynomial
  $l_1p(n)t+l_2p(n)$ is $(l_1t+l_2)\alpha$ and this is irrational
  if and only if $t\notin \mathbb{Q}/\alpha+\mathbb{Q}$. If this happens to be the case,
 then
$A_m\to 0$ and \eqref{E:fkl'} follows. Otherwise we have $t=q_1/\alpha+q_2$  for some $q_1,q_2\in \Q$. Since $t$ is assumed to be irrational we have $q_1\neq 0$ and $\alpha$ is irrational. Since the sequence $A_m$ is known to converge,  the limit in \eqref{E:fkl'} is equal to a constant times
\begin{equation}\label{E:fkl''}
\lim_{M\to\infty}\frac{1}{M}\sum_{m=1}^M
\ e\big(m(l_1q_1/\alpha+l_1q_2+l_2)\beta\big).
\end{equation}
We can rewrite $(l_1q_1/\alpha+l_1q_2+l_2)\beta$
as $(\tilde{q}_1/\alpha
+\tilde{q}_2)\beta$ for some $\tilde{q}_1, \tilde{q}_2\in \Q$ with $\tilde{q}_1\neq 0$.
Since $\alpha$ is irrational, this last expression is non-zero,  and as a consequence the limit
in \eqref{E:fkl''} is going to   be $0$ unless $(\tilde{q}_1/\alpha
+\tilde{q}_2)\beta$ is a non-zero integer. But this cannot be the case since by our assumption
$1/\beta \notin \mathbb{Q}/ \alpha+\mathbb{Q}$. Therefore, the limit
in \eqref{E:fkl''} is  $0$ completing the proof of \eqref{E:fkl}.

It remains to deal with  the case where  $t\in (0,1)$ is rational. For convenience we shall assume
that $t=1/k$ for some integer $k$ with $k\geq 2$ (if the numerator of $t$ is not $1$ the argument is similar).  Using a standard
approximation argument we can assume that $\phi(y)=e(ly)$ for some
$l\in\Z$. Then the limit in question becomes
$$
\lim_{M\to\infty}\frac{1}{M}\sum_{m=1}^M\Big( \frac{1}{N_m}
\sum_{n=n_m}^{n_m+N_m}
e\big((l+1/k)(p(n)+m\beta)-\{p(n)+m\beta\}/k\big)\Big).
$$
Arguing as in the proof of \eqref{E:fkl} we can show that the sequence $\big((l+1/k)(p(n)+m\beta),
p(n)+m\beta\big)$ is equidistributed  in the subset $H=\big\{\big((lk+1)x, kx\big)\colon x\in \T\big\}$ of $\T^2$ with respect to the sequence
$(S_M)_{M\in\N}$, where
$$S_M=\{(m,n)\in\Z^2\colon 1\leq m\leq M, \ n_m\leq n\leq n_m+N_m\}.$$
 Therefore, the last limit is equal to
$$
\int_0^1 e\big((lk+1)x- \{kx\}/k\big) \ dx=\sum_{i=0}^{k-1}
\int_{i/k}^{(i+1)/k}e\big((lk+1)x- (kx-i)/k\big)\ dx.
$$
Using the change of variables $x\to y+i/k$ we see that the last
expression is equal to
\begin{equation}\label{E:sum}
  \int_{0}^{1/k}e\big(lk y \big) \ dy \cdot \sum_{i=0}^{k-1} e((lk+1)i/k).
\end{equation}
Since $l\in \Z$ and $k\geq 2$ we have $lk+1\neq 0$.  It
follows that the sum appearing in \eqref{E:sum} is zero. This
completes the proof.
\end{proof}

\begin{proof}[Proof of Theorem~\ref{T:singlehardy}]
  %% We first prove the necessity of conditions $(i)$ and $(ii)$.
  %% Withought loss of generality we can assume that $c>0$.  If $(i)$
  %% holds then obviously $a(n)$ is bad for single recurrence. If
  %% $(ii)$ holds then for $\varepsilon$ small enough we have that
  %% $b-kc\notin [-\varepsilon,1+\varepsilon]$ for every $k\in\Z$. %%Let $e(x)=a(x)-cp(x)-b$, then $e(x)\to 0$ as $x\to\infty$.
  %% For every $k\in\Z$, large $n$, and $\varepsilon$ small enough we
  %% have
%% $$
%% [a(n)]\cdot 1/c-k=p(n)+\big(b-kc+e(n)-\{p(n)c+b\}\big)/c\notin
%% [-\varepsilon/c,\varepsilon/c]
%% $$
%% It follows that $a(n)$ is not good for recurrence for the rotation
%% by $1/c$ in the circle.

%% To prove the positive result
  If $|a(x)-c p(x)|\succ \log x$ for every $p\in \Z[x]$ and
  $c\in \R$, then by \cite{BKQW} the sequence $([a(n)])_{n\in\N}$
  is good for the ergodic theorem\footnote{This is an easy consequence of the spectral
    theorem and the following equidistribution result of Boshernitzan
    \cite{Bo3}: If $a\in \H$ has  polynomial growth and
    satisfies $|a(x)-p(x)|\succ\log x$ for every $p\in\Q[x]$, then the
    sequence $(a(n))_{n\in\N}$ is equidistributed in $\T$.},  and the
  result follows.

  %% {\bf Case 2.} Suppose that $a(x)= p(x)+b(x)$ for some $p\in
  %% \Q[x]$ and Hardy field function $b(x)$ that satisfies $1\prec
  %% b(x)\prec x$. Working with $a(kx)$ for some appropriate $k\in \N$
  %% we can assume that $p(x)\in \Z[x]$. Since $1\prec b(x)\prec x$
  %% for every $r\in\N$ we can find a sequence of intervals $I_{r,m}$,
  %% $m\in \N$, with $|I_{r,m}|\to \infty$ as $m\to\infty$, such that
  %% $[b(n)]\equiv 0 \pmod{r}$. It follows that for every $r\in \N$
  %% there exist $p_r\in Z[x]$ and $n_m\in\N$ such that
%%$$
%%\{rp_r(n), n_m\leq n\leq n_m+m\}\subset S.
%%$$

  Therefore,  we can assume that $a(x)= c p(x)+b(x)$ for some $p\in
  \Z[x]$, $c \in\R$, and $b\in \H$ that satisfies $1\prec
  b(x)\prec x$. Furthermore, we can assume that $b(x)\geq 0$ for
  large $x$ (the other case can be treated similarly).  If $c=0$, then
  the range of the sequence $([b(n)])_{n\in\N}$ contains all large
  enough integers (see Lemma~\ref{L:basic}) and so forms a set of recurrence.  If $c\neq 0$, let
  $\beta\in \R$ be such that $1/\beta \notin \mathbb{Q}/ c+\mathbb{Q}$.  Since $b(x)\prec x$ we have $b(x+1)-b(x)\to 0$ (see
  Lemma~\ref{L:basic}), and since $b(x)\to \infty$ it follows
  that there exist integers $N_m$ with $N_m\to\infty$ and intervals
  $I_m=[n_m,n_m+N_m]$, where $n_m\in\N$, such that
  \begin{equation}\label{E:J1}
    \lim_{m\to\infty}\sup_{n\in I_m}|b(n)-m\beta|= 0.
  \end{equation}
  Let
  \begin{equation}\label{E:J2}
    J=\{(m,n)\in\N^2\colon \{c p(n)+m\beta\}\in [1/2,3/4]\}.
  \end{equation}
  Consider the sequence $(S_M)_{M\in\N}$ where
$$
S_M=\{(m,n)\in \N^2\colon 1\leq m\leq M,\ n_m\leq n\leq n_m+N_m \}.
$$
Notice that because of \eqref{E:J1} and \eqref{E:J2}, for $(m,n)\in
S_M\cap J$ with $m$ big enough, we have $[c p(n)+b(n)]=[c
p(n)+m\beta]$.  Applying Lemma~\ref{L:t} for $\phi={\bf
  1}_{[1/2,3/4]}$, and noticing that the set $J$ has positive density
($=1/4$) with respect to the sequence $(S_M)_{M\in\N}$, we have for
every $t\in (0,1)$ that
$$
\lim_{M\to\infty}\frac{1}{|S_M|}\sum_{(m,n)\in S_M\cap J}
e\big([c
p(n)+b(n)]t\big)=\lim_{M\to\infty}\frac{1}{|S_M|}\sum_{(m,n)\in
  S_M\cap J} e\big([c p(n)+m\beta]t\big)=0.
$$
Using this and the spectral theorem we get that
$$
\lim_{M\to\infty}\frac{1}{|S_M|}\sum_{(m,n)\in S_M\cap J} \int f \cdot
T^{[c p(n)+b(n)]}f \ d\mu= \Big(\int f \ d\mu\Big)^2
$$
for every $f\in L^{\infty}(\mu)$. This implies the result and
completes the proof.
\end{proof}

\end{document}